\documentclass{amsart}
\usepackage{amsmath,amssymb}
\usepackage{pstricks, pst-plot, pst-node, pst-grad, pst-math}
\usepackage{graphicx,xcolor}
\usepackage{arydshln}

 \newtheorem{thm}{Theorem}[section]
 
 \newtheorem{lem}[thm]{Lemma}
 \newtheorem{prop}[thm]{Proposition}
 \theoremstyle{definition}
 \newtheorem{defn}[thm]{Definition}
 \theoremstyle{remark}
 \newtheorem{rem}[thm]{Remark}
 \newtheorem{ex}[thm]{Example}
 \numberwithin{equation}{section}
 \usepackage{xcolor}
 \usepackage{mathabx}
 
\newcommand{\om}[1]{\mathcal{{ #1}}}	
\newcommand{\mon}[1]{\widehat #1}		
\newcommand{\side}[1]{\widetilde #1}	

\newcommand{\Hs}{\mathcal{H}}	
\newcommand{\HS}{\widetilde{\mathcal{H}}}	
\newcommand{\Hss}{\widehat{\mathcal{H}}}	
\newcommand{\hs}{\widecheck{\mathcal{H}}}	
	
\newcommand{\Ls}{\mathcal{L}}
\newcommand{\Bs}{\mathcal{B}}

\newcommand{\R}{\mathbb{R}}
\newcommand{\C}{\mathbb{C}}

\newcommand{\dom}{\mathop{\mathcal{D}}}
\newcommand{\Cdom}{\mathop{\Omega}}
\newcommand{\delt}{\widehat{\delta}}

\begin{document}

\title[On Equivalence of Operator Matrix Functions]{On Equivalence and Linearization of Operator Matrix Functions with Unbounded Entries}

\author{Christian Engstr\"om}
\address{Department of Mathematics and Mathematical Statistics, Ume\aa \ University, SE-901 87 Ume\aa, Sweden}
\email{christian.engstrom@math.umu.se}
\author{Axel Torshage} 
\address{Department of Mathematics and Mathematical Statistics, Ume\aa \ University, SE-901 87 Ume\aa, Sweden}
\email{axel.torshage@math.umu.se}



\date{}

\begin{abstract}
In this paper we present equivalence results for several types of unbounded operator functions. A generalization of the concept equivalence after extension is introduced and used to prove equivalence and linearization for classes of unbounded operator functions. Further, we deduce methods of finding equivalences to operator matrix functions that utilizes equivalences of the entries. Finally, a method of finding equivalences and linearizations to a general case of operator matrix polynomials is presented.
\end{abstract}

\maketitle
\section{Introduction}
Spectral properties of unbounded operator matrices are of major interest in operator theory and its applications \cite{book}. Important examples are systems of partial differential equations with $\lambda$-dependent coefficients or boundary conditions \cite{MR1042212,MR1815190,MR1354980,MR3543766,ATEN}.  A concept of equivalence can be used to compare spectral properties of different  operator functions and the problem of classifying bounded analytic operator functions modulo equivalence has been studied intensely \cite{MR0482317,MR0473894,MR2363355,MR634171}. The properties preserved by equivalences include the spectrum and for holomorphic operator functions there is a one-to-one correspondence between their Jordan chains,  \cite[Proposition 1.2]{MR1155350}. Our aim is to generalize some of the results in those articles and study a concept of equivalence for classes of operator functions whose values are unbounded linear operators. A prominent result in this direction is the equivalence between an operator matrix and its Schur complements \cite{MR1285306,MR1382107,book}.  

In this paper, we consider systems described by $n\times n$ operator matrix functions and study a concept of equivalence when some of the entries are Schur complements, polynomials, or can be written as a product of operator functions. Examples of this type are the operator matrix function with quadratic polynomial entries that were studied in \cite{MR1911850} and functions with rational and polynomial entries in plasmonics \cite{Mortensen}. In order to extend previous results to cases with unbounded entries, we generalize in definition \ref{1eaet} the concept of equivalence after extension in \cite{MR0482317}. This new concept can be used to compare spectral properties of two unbounded operator functions, but also for determining the correspondence between the domains and when two operator functions are simultaneously closed.
 Our main results are (i) equivalence results for operator matrix functions containing unbounded Schur complement entries (Theorem \ref{2schurrat0}) and polynomial entries (Theorem \ref{2genP1lin}) and
(ii) a systematic approach to linearize operator matrix functions with polynomial entries (Theorem \ref{polygenlin} together with the algorithm in Proposition \ref{polyrref} or Proposition \ref{polyrref2}).

Throughout this paper, $\Hs$ with or without subscripts, tildes, hats, or primes denote complex Banach spaces. Moreover, $\Ls(\Hs,\HS)$ denotes the collection of linear (not necessarily bounded) operators between $\Hs$ and $\HS$. The space of everywhere defined bounded operators between $\Hs$ and $\HS$ is denoted $\Bs(\Hs,\HS)$ and we use the notations $\Ls(\Hs):=\Ls(\Hs,\Hs)$ and $\Bs(\Hs):=\Bs(\Hs,\Hs)$. For convenience, a product Banach space of $d$ identical Banach spaces is denoted
\[
	\Hs^d:=\bigoplus_{i=1}^d \Hs,\quad \text{where}\quad \Hs^d:=\{0\}\text{ for }d\leq0.
\]
The domain of an operator $A\in\Ls(\Hs,\HS)$ is denoted $\dom (A)$ and if $A$ is closable the closure of $A$ is denoted $\overline{A}$. In the following, we denote for a linear operator $A$ the spectrum and resolvent set by $\sigma(A)$ and $\rho(A)$, respectively. The point spectrum $\sigma_{p}(A)$, continuous spectrum $\sigma_{c}(A)$, and residual spectrum $\sigma_{r}(A)$ are defined as in \cite[Section I.1]{MR929030}. 

Let $\Cdom\subset \C$ be a non-empty open set and let $T:\Cdom\rightarrow \Ls(\Hs,\Hs')$ denote an operator function. Then the spectrum of $T$ is
\[
	\sigma(T):=\{\lambda\in\Cdom\,:\, 0\in\sigma (T(\lambda))\}.
\]
An operator matrix function $\om{T}:\Cdom\rightarrow \Ls(\Hs\oplus\HS,\Hs'\oplus\HS')$ have a representation  as 
\[
	\om{T}(\lambda):=\begin{bmatrix}
		A(\lambda) &B(\lambda)\\
		C(\lambda) &D(\lambda)
	\end{bmatrix}, \quad \lambda\in\Cdom.
\]
 Unless otherwise stated the \emph{natural domain}
\[ 
	\dom(\om{T}(\lambda)):=\dom(A(\lambda))\cap\dom(C(\lambda))\oplus\dom(B(\lambda))\cap\dom(D(\lambda)),\quad \lambda\in\Cdom
\] 
is assumed \cite[Section 2.2]{book}. 

The paper is organized as follows. In Section $2$ we generalize concepts of equivalence to study functions whose values are unbounded operators.
In particular, the concept \emph{equivalence after operator function extension} is defined, which enable us to show an equivalence for pairs of unbounded operator functions. 
We provide natural generalizations of results that for bounded operator functions are well known. Further, we show how equivalence for an entry in an operator matrix function can be used to find an equivalence for the full operator matrix function.
 
Section $3$ contains three subsections, one for each of the studied equivalences: 
Schur complements, \cite{book,MR990188,MR1285306,MR3543766},  multiplication of operator functions, \cite{MR0482317}, and operator polynomials, \cite{MR511976,MR971506}, each structured similarly. First, an equivalence for the class of operator functions is presented and then we show how this equivalence can be used to prove equivalences for operator matrix functions. 
 
In Section $4$ we use the results from Section $3$ to also find equivalences between a class of operator matrix functions and operator matrix  polynomials. Moreover, we discuss two different ways of finding linear equivalences (linearizations) of operator matrix polynomials. The section is concluded with an example on how the results from Section $3$ and Section $4$  can be used jointly to linearize  operator matrix functions.
\section{Equivalence and equivalence after operator function extension}

In this section we introduce the concepts used to classify unbounded operator functions up to equivalence. These concepts were used to study bounded operator functions \cite{MR0482317,MR2151934} and we present natural generalizations to the unbounded case.

Let $\Cdom_S,\Cdom_T\subset\C$ and consider the operator functions $S:\Cdom_S\rightarrow \Ls(\Hs,\Hs')$ and $T:\Cdom_T\rightarrow \Ls(\Hss,\Hss')$ with domains $\dom (S)$ and $\dom (T)$, respectively. Then $S$ and $T$ are called \emph{equivalent} on $\Cdom\subset\Cdom_S\cap\Cdom_T$ if there exist operator functions 
$E:\Cdom\rightarrow\Bs(\Hss',\Hs')$ and $F:\Cdom\rightarrow\Bs(\Hs,\Hss)$ invertible for $\lambda\in\Cdom$ such that 
\begin{equation}\label{1eqdef}
	S(\lambda)=E(\lambda)T(\lambda)F(\lambda),\quad \dom(S(\lambda))=F(\lambda)^{-1}\dom(T(\lambda)).
\end{equation}
 It can easily be verified that \eqref{1eqdef} is an equivalence relation.  
 
 Many of the results in this paper are point-wise, i.e. for a fixed operator. However, the results on spectral properties and the main results on linearization in Section 3.3 and in Section 4 require dependence on $\lambda$. For consistency, we therefore state all theorems for operator functions.
The following proposition is immediate from its construction \cite{MR1382107}, \cite[Lemma 2.3.2]{book}.
\begin{prop}\label{1defcor} 
Assume that $S:\Cdom_S\rightarrow \Ls(\Hs,\Hs')$ is equivalent to $T:\Cdom_T\rightarrow \Ls(\Hss,\Hss')$ on $\Cdom\subset \Cdom_S\cap \Cdom_T$, and let $E$ and $F$ denote the operator functions in the equivalence relation \eqref{1eqdef}. Then the operator $S(\lambda)$ is closed (closable) for $\lambda \in \Cdom$ if and only if $T(\lambda)$ is closed (closable), where the closure of a closable $S(\lambda)$ is
\[
	\overline{S(\lambda)}=E(\lambda)\overline{T(\lambda)}F(\lambda),\quad \dom(\overline{S(\lambda)})=F^{-1}(\lambda)\dom( \overline{T(\lambda)}).
\]
Let $S_{\Cdom}$ and $T_{\Cdom}$ denote the restrictions of  $S$ and $T$ to $\Cdom$. Then
\[
\sigma(\overline{T}_{\Cdom})=\sigma(\overline{S}_{\Cdom}),\ \sigma_p(\overline{T}_{\Cdom})=\sigma_p(\overline{S}_{\Cdom}),\ \sigma_c(\overline{T}_{\Cdom})=\sigma_c(\overline{S}_{\Cdom}),\ \sigma_r(\overline{T}_{\Cdom})=\sigma_r(\overline{S}_{\Cdom}).
\]
\end{prop}
Gohberg et al. \cite{MR0482317} and Bart et al. \cite{MR2151934} studied a generalization of equivalence  called equivalence after extension. Here, we introduce a more general definition of equivalent after extension, which we for clarity call \emph{equivalence after operator function extension}. 
\begin{defn}\label{1eaet}
Let  $S:\Cdom_S\rightarrow \Ls(\Hs,\Hs')$ and $T:\Cdom_T\rightarrow \Ls(\Hss,\Hss')$ denote operator functions with domains $\dom (S)$ and $\dom (T)$, respectively. Assume there are operator functions $W_S:\Cdom\rightarrow \Ls(\hs_S,\hs_S)$ and $W_T:\Cdom\rightarrow \Ls(\hs_T,\hs_T)$ invertible on $\Cdom\subset\Cdom_S\cap\Cdom_T$ such that
\[
	\begin{array}{l l}
		S(\lambda)\oplus W_S(\lambda), & \dom(S(\lambda)\oplus W_S(\lambda))=\dom(S(\lambda))\oplus\dom(W_S(\lambda)),\vspace{5pt} \\
		T(\lambda)\oplus W_T(\lambda), & \dom(T(\lambda)\oplus W_T(\lambda))=\dom(T(\lambda))\oplus\dom(W_T(\lambda)),
	\end{array}
\]
are equivalent on $\Cdom$.  Then $S$ and $T$ are said to be \emph{equivalent after operator function extension on $\Cdom$}.
 The operator functions $S$ and $T$ are said to be \emph{equivalent after one-sided operator function extension} on $\Cdom$ if either $\hs_S$ or $\hs_T$ can be chosen to $\{0\}$. If $\hs_T$ can be chosen to $\{0\}$ then we say that $S$ is after $W_S$-\emph{extension} equivalent to $T$ on $\Cdom$.
\end{defn}

The definition of equivalent after extension in \cite{MR2151934} correspond in Definition \ref{1eaet} to the case $W_S(\lambda)= I_{\hs_S}$ and  $W_T(\lambda)= I_{\hs_T}$ for all $\lambda\in\Cdom$. We allow $W_S$ and $W_T$ to be unbounded operator functions and can therefore study a concept of equivalence for a larger class of unbounded operator function pairs $S$ and $T$. 

In particular, the equivalence results for Schur complements and polynomial problems presented in Section \ref{sec:rat} respectively Section \ref{sec:poly}, can not be described by an equivalence after extension with the identity operator. In the equivalence results for multiplication operators in Section \ref{sec:mult} the operator function $W$ is bounded (actually $W(\lambda)= I$ for all $\lambda\in \C$). Thus, in that case the standard definition of equivalence after extension is sufficient as well.

Proposition \ref{1defcor} shows that two equivalent unbounded operator functions have the same spectral properties and it provides the correspondence between the domains. In the following proposition, those results are extended to include operator functions that are equivalent after operator function extension.
\begin{prop}\label{2defcor} 
Assume that $S:\Cdom_S\rightarrow \Ls(\Hs,\Hs')$ and $T:\Cdom_T\rightarrow \Ls(\Hss,\Hss')$, are equivalent after operator function extension on $\Cdom\subset \Cdom_S\cap\Cdom_T $. Let $W_S:\Cdom\rightarrow \Ls(\hs_S,\hs_S)$ and $W_T:\Cdom\rightarrow \Ls(\hs_T,\hs_T)$ denote the invertible operator functions such that $S(\lambda)\oplus W_S(\lambda)$ is equivalent to $T(\lambda)\oplus W_T(\lambda)$ for $\lambda\in\Cdom$ and let $E$, $F$ be the operator functions in the equivalence relation \eqref{1eqdef}. Define the operator
$\pi_{\Hs'}:\Hs'\oplus \hs_S\rightarrow \Hs'$ as $\pi_{\Hs'}u\oplus v=u$ and let $\tau_{\Hs}$ denote the natural embedding of $\Hs$ into $\Hs\oplus \hs_S$ given by $\tau_\Hs u=u\oplus0_{\hs_S}$. Then for $\lambda\in\Cdom$ we have the relations
\[
	\begin{array}{l}
	S(\lambda) =\pi_{\Hs'}E(\lambda)\begin{bmatrix}
	T(\lambda) & \\
	& W_T(\lambda)
	\end{bmatrix}F(\lambda)\tau_{\Hs},\\
	\dom(S(\lambda))=\pi_{\Hs}F^{-1}(\lambda)(\dom(T(\lambda))\oplus\dom(W_T(\lambda))),
	\end{array}
\]
and the operator $S(\lambda)$ is closed (closable) if and only if $T(\lambda)$ is closed (closable). The closure of a closable operator $S(\lambda)$ is
\[
	\begin{array}{l l}
	\overline{S(\lambda)}=\pi_{\Hs'}E(\lambda)\begin{bmatrix}
	\overline{T(\lambda)} & \\
	& W_T(\lambda)
	\end{bmatrix}F(\lambda)\tau_{\Hs},\\
	\dom(\overline{S(\lambda)})=\pi_{\Hs}F^{-1}(\lambda)(\dom(\overline{T(\lambda)})\oplus\dom(W_T(\lambda))),\\
	\end{array}
\] 
and we have then
\[
\sigma(\overline{T}_{\Cdom})=\sigma(\overline{S}_{\Cdom}),\ \sigma_p(\overline{T}_{\Cdom})=\sigma_p(\overline{S}_{\Cdom}),\ \sigma_c(\overline{T}_{\Cdom})=\sigma_c(\overline{S}_{\Cdom}),\ \sigma_r(\overline{T}_{\Cdom})=\sigma_r(\overline{S}_{\Cdom}),
\]
where $S_\Omega$ and $T_\Omega$ denote the restrictions of $S$ and $T$ to $\Omega$.
\end{prop}
\begin{proof}
From Definition \ref{1eaet} it follows that for $\lambda\in\Cdom$ the following relations hold
\[
	\begin{array}{l}
	\begin{bmatrix}
	S(\lambda) & \\
	& W_S(\lambda)
	\end{bmatrix}=E(\lambda)\begin{bmatrix}
	T(\lambda) & \\
	& W_T(\lambda)
	\end{bmatrix}F(\lambda), \\
	\dom(S(\lambda)\oplus W_S(\lambda))=F^{-1}(\lambda)(\dom(T(\lambda))\oplus\dom(W_T(\lambda))).
	\end{array}
\]
The result then follows from Proposition \ref{1defcor} and that the closure of a block  diagonal operator coincides with the closures of the blocks. 
\end{proof}

Below we show how an equivalence for an entry in an operator matrix function can be used to find an equivalence for the full operator matrix function.
A general operator matrix function $\widehat{\om{S}}:\Cdom\rightarrow \Ls\left(\bigoplus_{i=1}^n \Hs_i\rightarrow\bigoplus_{i=1}^n \Hs_i'\right)$ defined on its natural domain can be represented as
\begin{equation}\label{2genT}
\widehat{\om{S}}(\lambda):=\begin{bmatrix}
S_{1,1}(\lambda) &\hdots&S_{1,n}(\lambda) \\
\vdots & \ddots &\vdots\\
S_{n,1}(\lambda) & \hdots &S_{n,n}(\lambda)
\end{bmatrix},\quad \lambda\in\Cdom.
\end{equation}
However, any entry $S(\lambda):=S_{j,i}(\lambda)$ can be moved to the upper left corner by changing the orders of the spaces, which result in the equivalent problem
\begin{equation}\label{2genT2}
\begin{bmatrix}
S(\lambda) &\hdots \\
\vdots & \ddots 
\end{bmatrix}=\begin{bmatrix}
S(\lambda) &X(\lambda) \\
Y(\lambda) &Z(\lambda)
\end{bmatrix}=:\om{S}(\lambda).
\end{equation}
Hence, it is sufficient to study the $2\times2$ system given in \eqref{2genT2},
where $S:\Cdom\rightarrow \Ls(\Hs,\Hs')$, $X:\Cdom\rightarrow \Ls(\HS,\Hs')$, $Y:\Cdom\rightarrow \Ls(\Hs,\HS')$ and $Z:\Cdom\rightarrow \Ls(\HS,\HS')$.

\begin{lem}\label{2partblin}
Assume that $S:\Cdom_S\rightarrow\Ls(\Hs,\Hs')$ is equivalent to $T:\Cdom_T\rightarrow\Ls(\Hss,\Hss')$ on $\Cdom\subset\Cdom_S\cap\Cdom_T$. Let $E:\Cdom\rightarrow\Bs(\Hss',\Hs')$ and $F:\Cdom\rightarrow\Bs(\Hs,\Hss)$ be the operator functions invertible for $\lambda\in\Cdom$, such that  $S(\lambda)=E(\lambda)T(\lambda)F(\lambda)$.
Consider $\om{S}(\lambda)$ defined in \eqref{2genT2} 
and let $\side{E}:\Cdom\rightarrow\Bs(\Hss',\HS')$, $\side{F}:\Cdom\rightarrow\Bs(\HS,\Hss)$  be a solution pair of 
\begin{equation}\label{3systcond}
	\side{E}(\lambda)E(\lambda)^{-1}X(\lambda)+Y(\lambda)F(\lambda)^{-1}\side{F}(\lambda)-\side{E}(\lambda)T(\lambda)\side{F}(\lambda)=0, \quad\lambda\in\Cdom.
\end{equation}
 Then $\om{S}$ is equivalent to $\om{T}:\Cdom\rightarrow\Ls(\Hss\oplus\HS,\Hss'\oplus\HS')$ on $\Cdom$, where
 \[
	\om{S}(\lambda)=\om{E}(\lambda)\om{T}(\lambda)\om{F}(\lambda),\quad \dom(\om{S}(\lambda))=\om{F}^{-1}(\lambda)\dom(\om{T}(\lambda)),
\]
with
\[
\om{T}(\lambda):=
	\begin{bmatrix}
	T(\lambda) & E^{-1}(\lambda)X(\lambda)-T(\lambda)\side{F}(\lambda)\\
	Y(\lambda)F^{-1}(\lambda)-\side{E}(\lambda)T(\lambda) & Z(\lambda)
	\end{bmatrix},
\]
and 
\[
\om{E}(\lambda):=
	\begin{bmatrix}
	 E(\lambda)&  \\
	 \side{E}(\lambda)& I_{\HS'}\\
	\end{bmatrix},\quad
\om{F}(\lambda):=
	\begin{bmatrix}
	F(\lambda)&\side{F}(\lambda) \\
	 &I_{\HS}\\
	\end{bmatrix}.
\]
\end{lem}
\begin{proof}
Under the assumption \eqref{3systcond}, the lemma follows immediately by verifying $\om{S}(\lambda)=\om{E}(\lambda)\om{T}(\lambda)\om{F}(\lambda)$.
\end{proof}

\begin{rem}
The condition \eqref{3systcond} is satisfied in the trivial case $\side{E}=0$, $\side{F}=0$, and for the problems we study in Section \ref{Sec3}. A similar result holds also when \eqref{3systcond} is not satisfied, but then the $(2,2)$-entry in $\om{T}(\lambda)$ will not be of the same form.

\end{rem}
\section{Equivalences for classes of operator matrix functions}\label{Sec3}
In this section, we study  Schur complements, operator functions consisting of multiplications of operator functions, and operator polynomials. Each type will be studied similarly: First an equivalence after operator function extension is shown, which then together with Lemma \ref{2partblin} is utilized in an operator matrix function.

\begin{rem}\label{3rem}
Assume that $S(\lambda)\oplus W(\lambda)$ is equivalent to $T(\lambda)$ for $\lambda\in\Cdom$ and let $\om{S}$ be defined as \eqref{2genT2}. For the equivalence relation between $T$ and $\om{S}$ we want the block $S(\lambda)\oplus W(\lambda)$ intact to be able to apply Lemma \ref{2partblin} directly. Therefore, an equivalence after $W$-extension of $\om{S}(\lambda)$ is given as 
\begin{equation}\label{3remlin}
	\begin{bmatrix}
	S(\lambda)& &X(\lambda)\\
	& W(\lambda)\\
	Y(\lambda)& &Z(\lambda)
	\end{bmatrix}=\begin{bmatrix}
	I\\& &I\\& I
	\end{bmatrix}\begin{bmatrix}
	S(\lambda)& X(\lambda)\\
	Y(\lambda)& Z(\lambda)\\
	&& W(\lambda)
	\end{bmatrix}\begin{bmatrix}
	I\\& &I\\& I
	\end{bmatrix},
\end{equation}
instead of $\om{S}(\lambda)\oplus W(\lambda)$. 
\end{rem}

\subsection{Schur complements}\label{sec:rat}
Let  $D:\Cdom_D\rightarrow\Ls(\hs)$ denote an operator function with domain  $\dom (D(\lambda))$ for $\lambda\in\Cdom_D\subset\C$. Assume that $\Cdom'\subset\Cdom_D\cap\rho(D)$ is non-empty and let ${S}:\Cdom'\rightarrow\Ls(\Hs,\Hs')$ for $\lambda\in\Cdom' $ be defined as
\begin{equation}\label{3eq:rat}
{S}(\lambda):=A(\lambda)-B(\lambda)D(\lambda)^{-1}C(\lambda),\quad \dom( {S}(\lambda)): =\dom (A(\lambda))\cap \dom (C(\lambda)), 
\end{equation}
where $A:\Cdom'\rightarrow\Ls(\Hs,\Hs')$,  $B:\Cdom'\rightarrow\Ls(\hs,\Hs')$, $C:\Cdom'\rightarrow\Ls(\Hs,\hs)$, and  $\dom (D(\lambda)) \subset \dom (B(\lambda))$. The claims in the following lemma are standard results for Schur complements \cite{MR1382107}, \cite[Theorem 2.2.18]{book} formulated in terms of an equivalence after operator function extension. For convenience of the reader we provide a short proof.

\begin{lem}\label{2Schur} 
Let the operator ${S}(\lambda)$ denote the operator defined in \eqref{3eq:rat}, assume that $C(\lambda)$ is densely defined in $\Hs$, and that $D^{-1}(\lambda)C(\lambda)$ is bounded on $\dom (C(\lambda))$ for all $\lambda\in\Cdom'$. Define the operator matrix function $T$ on its natural domain as
\[
	T(\lambda):=\begin{bmatrix}
	A(\lambda) & B(\lambda)\\
	C(\lambda) & D(\lambda)
	\end{bmatrix},\quad \lambda\in{\Cdom}'. 
\]
Then  ${S}$ is after $D$-extension equivalent to $T$ on $\Cdom'$, where the operator matrix functions $E$ and $F$ in the equivalence relation \eqref{1eqdef} are 
\[
 E(\lambda):=\begin{bmatrix}
	I_{\Hs'}& -B(\lambda)D(\lambda)^{-1}\\
	& I_{\hs}
	\end{bmatrix},\quad
	F(\lambda):=\begin{bmatrix}
	I_{\Hs}& \\
	-\overline{D(\lambda)^{-1}C(\lambda)} & I_{\hs}
	\end{bmatrix}.
\]
The operator $T(\lambda)$ is closable if and only if ${S}(\lambda)$ is closable, and
\[\begin{array}{l}
\overline{T}(\lambda)=
	\begin{bmatrix}
	\overline{{S}(\lambda)}+B(\lambda)\overline{D(\lambda)^{-1}C(\lambda)} &B(\lambda) \\
	D(\lambda)\overline{D(\lambda)^{-1}C(\lambda)}&D(\lambda)
	\end{bmatrix},\\
	\dom(\overline{T(\lambda)})=\{(u,v)\in\Hs\oplus\hs:u\in \dom(\overline{{S}(\lambda)}),\overline{D(\lambda)^{-1}C(\lambda)}u+v\in\dom(D(\lambda))\}.
	\end{array}
\]
\end{lem}

\begin{proof}
The operators matrices $E(\lambda)$ and $F(\lambda)$ are bounded on $\dom(C(\lambda))$ and \\$\overline{D(\lambda)^{-1}C(\lambda)}=D(\lambda)^{-1}C(\lambda)$ on $\dom ({S}(\lambda))$. The result then follows from the factorization
\[
	\begin{bmatrix}
	{S}(\lambda) & \\
	& D(\lambda)
	\end{bmatrix}=\begin{bmatrix}
	I_{\Hs'}& -B(\lambda)D(\lambda)^{-1}\\
	& I_{\hs}
	\end{bmatrix}\begin{bmatrix}
	A(\lambda) & B(\lambda)\\
	C(\lambda) & D(\lambda)
	\end{bmatrix}\begin{bmatrix}
	I_{\Hs}& \\
	-D(\lambda)^{-1}C(\lambda) & I_{\hs}
	\end{bmatrix}
\]
and Proposition \ref{2defcor}.
\end{proof}
\begin{rem}
If $D$ is unbounded, ${S}$  and $T$ are not equivalent after extension. However, they are  equivalent after $D$-extension. 
\end{rem} 
The domain and the closure are not explicitly stated in the equivalences in the remaining part of the article but they can be derived using the relations in Proposition \ref{2defcor}.

\begin{thm}\label{2schurrat0}
Let ${S}$, $E$, and $F$ denote the operator functions  on $\Cdom'\supset\Cdom$ defined in Lemma \ref{2Schur}. The operator matrix function $\om{{S}}:\Cdom\rightarrow\Ls(\Hs\oplus\HS,\Hs'\oplus\HS')$ is on its natural domain  defined as
\[
	\begin{array}{l}
	\om{{S}}(\lambda):=\begin{bmatrix}
	{S}(\lambda)&X(\lambda)\\
	Y(\lambda)&Z(\lambda)
	\end{bmatrix},\quad \lambda\in \Cdom.
	\end{array}
\]
Define the operator matrix function $\om{T}:\Cdom\rightarrow\Ls(\Hs\oplus\hs\oplus\HS,\Hs'\oplus\hs'\oplus\HS)$ by 
\[
	\om{T}(\lambda):=
	\begin{bmatrix}
	A(\lambda)& B(\lambda) & X(\lambda)\\
	C(\lambda) &D(\lambda)& \\
	Y(\lambda) &  &Z(\lambda)
	\end{bmatrix},\quad \lambda\in\Cdom.
\]
Then, $\om{{S}}$ is after $D$-extension with respect to structure \eqref{3remlin}  equivalent to $\om{T}$ on $\Cdom$, where the operator matrix functions $\om{E}$ and $\om{F}$ in the equivalence relation \eqref{1eqdef} for $\lambda\in \Cdom$  are 
\[
\om{E}(\lambda):=
\begin{bmatrix}
E(\lambda)\\
& I_{\HS'}\\
\end{bmatrix},\quad
\om{F}(\lambda):=
\begin{bmatrix}
F(\lambda)\\
& I_{\HS}\\
\end{bmatrix}.
\]
\end{thm}
\begin{proof}
From Lemma \ref{2Schur}, it follows that ${S}(\lambda)\oplus D(\lambda)=E(\lambda)T(\lambda)F(\lambda)$. By using Lemma \ref{2partblin} with 
$\side{E}=0$ and $\side{F}=0$, the proposed $\mathcal{E}(\lambda)$ and $\mathcal{F}(\lambda)$ are obtained and 
\[
	\om{T}(\lambda)=
	\begin{bmatrix}
	\begin{array}{c c}
	A(\lambda)& B(\lambda) \\
	C(\lambda)& D(\lambda)
	\end{array}
	& E(\lambda)^{-1}\begin{bmatrix} X(\lambda)\\ 0 \end{bmatrix}\\
	\begin{bmatrix} Y(\lambda)& 0 \end{bmatrix}F^{-1}(\lambda)  & Z(\lambda)\\
	\end{bmatrix}=
	\begin{bmatrix}
	A(\lambda)& B(\lambda) & X(\lambda)\\
	C(\lambda)& D(\lambda)& \\
	Y(\lambda) &   & Z(\lambda)\\
	\end{bmatrix}.
\]
\end{proof}
			
\subsection{Products of operator functions}\label{sec:mult}
Assume that for some $n\in\mathbb{N}$ the operator $M:\Cdom'\rightarrow\Bs(\Hs_n,\Hs_0)$ can be written as
\begin{equation}\label{3polyop}
	M(\lambda):=M_1(\lambda)M_2(\lambda)\hdots M_n(\lambda),\quad \lambda\in{\Cdom}',
\end{equation}
where $M_k:\Cdom'\rightarrow\Bs(\Hs_k,\Hs_{k-1})$. The following lemma is a straightforward generalization of a result in \cite{MR0482317}.
\begin{lem}\label{3jordig}
Let $M$ denote the operator function \eqref{3polyop} and set $\Hs:=\oplus_{k=1}^{n-1}\Hs_k$. Define the operator matrix function $T:\Cdom'\rightarrow\Bs(\Hs\oplus\Hs_n,\Hs_0\oplus\Hs)$ as
\[
T(\lambda):=
	\begin{bmatrix}
	M_1(\lambda)&\\
	-I_{\Hs_1}& \ddots\\
	&\ddots & \ddots\\
	& & -I_{\Hs_{n-1}}&M_n(\lambda)
	\end{bmatrix},\quad \lambda\in{\Cdom}'.
\]
Then $M$ is after  $I_{\Hs}$-extension equivalent to $T$, where the operator matrix functions $E:\Cdom'\rightarrow\Bs(\Hs_0\oplus\Hs)$ and $F:\Cdom'\rightarrow\Bs(\Hs\oplus\Hs_n)$ in the equivalence relation \eqref{1eqdef} are
\[
	\begin{array}{l}
	E(\lambda):=\begin{bmatrix}
	I_{\Hs_0}&M_1(\lambda)&\hdots& \prod_{k=1}^{n-1} M_k(\lambda)\\
	&\ddots &\ddots&\vdots\\
	& & \ddots &M_{n-1}(\lambda)\\
	& & &I_{\Hs_{n-1}}
	\end{bmatrix},\\
	F(\lambda):=\begin{bmatrix}
	\prod_{k=2}^{n} M_k(\lambda)&-I_{\Hs_1}\\
	\vdots& & \ddots\\
	M_n(\lambda)& & &-I_{\Hs_{n-1}}\\
	I_{\Hs_n}
	\end{bmatrix}.
	\end{array}
\]
\end{lem}
\begin{proof}
For $n=2$ the equivalence result is used in the proof of \cite[Theorem 4.1]{MR0482317} and the claims in the lemma follows by applying that equivalence iteratively. 
\end{proof}
\begin{rem}
Consider the operator function \eqref{3polyop} with $n=2$ and write $M(\lambda)$ in the form 
\[
	M(\lambda)=-M_1(\lambda)(-I_{\Hs_1})^{-1}M_2(\lambda).
\]
Then, Lemma \ref{2Schur} can be used to obtain the same equivalence result as in Lemma \ref{3jordig}. Doing this iteratively for $n>2$ shows that Lemma \ref{3jordig} is a consequence of Lemma \ref{2Schur}. However, $M(\lambda)$ is an important case that has been studied separately (see e.g. \cite[Theorem 4.1]{MR0482317}).
\end{rem}

Below we show how Lemma \ref{3jordig} can be applied to an operator matrix function.

\begin{thm}\label{3jordigblo}
Let $M$, $E$, and $F$ denote the operator functions on $\Cdom'\supset\Cdom$ defined  in Lemma \ref{3jordig}.
The operator matrix function $\om{M}:\Cdom\rightarrow\Ls(\Hs_n\oplus\HS,\Hs_0\oplus\HS')$ is on its natural domain  defined as
\[
\om{M}(\lambda):=\begin{bmatrix}
M(\lambda)& X(\lambda)\\
Y(\lambda)& Z(\lambda)
\end{bmatrix},\quad \lambda\in\Cdom. 
\]
Then $\om{M}$ is after $I_{\Hs}$-extension, with respect to the structure \eqref{3remlin}, equivalent to $\om{T}:\Cdom\rightarrow \Ls(\Hs\oplus\Hs_n\oplus\HS,\Hs_0\oplus\Hs\oplus\HS')$, which on its natural domain is defined as

\[
	\begin{array}{l}
\om{T}(\lambda):=
	\begin{bmatrix}
	M_1(\lambda)& & &  & X(\lambda)\\
	 -I_{\Hs_1}&M_2(\lambda)\\
	 & \ddots& \ddots\\
	& & -I_{\Hs_{n-1}}&M_n(\lambda)\\
	& & & Y(\lambda)& Z(\lambda)
	\end{bmatrix},\quad \lambda\in\Cdom.
	\end{array}
\]
 The operator matrix functions $\om{E}:\Cdom\rightarrow\Bs(\Hs_0\oplus\Hs\oplus\HS')$ and $\om{F}:\Cdom\rightarrow\Bs(\Hs\oplus\Hs_n\oplus\HS)$ in the equivalence relation \eqref{1eqdef} are
\[
	\om{E}(\lambda):=\begin{bmatrix}
	E(\lambda)\\�
	& I_{\HS'}
	\end{bmatrix},\quad 
	\om{F}(\lambda):=\begin{bmatrix}
	F(\lambda)\\�
	& I_{\HS}
	\end{bmatrix}.
\]
\end{thm}
\begin{proof}
The claims follow by combining the extension in Lemma \ref{3jordig} with Lemma \ref{2partblin} for the case $\side{E}(\lambda)=0$,
$\side{F}(\lambda)=0$. This derivation is similar to the proof of Theorem \ref{2schurrat0}.
\end{proof}

\subsection{Operator polynomials}\label{sec:poly}
Let $l\in\{0,\hdots,d\}$ and consider the operator polynomial $P:\C\rightarrow \Ls(\Hs)$,
\begin{equation}\label{2polop}
	P(\lambda):=\sum_{i=0}^{d} \lambda^iP_i,\quad \dom (P(\lambda)):=\dom (P_l), \quad \lambda\in\C, 
\end{equation}
where $P_i\in\Bs(\Hs)$ for $i\neq l$. A linear equivalence is for $l=0$ in principal given by \cite[p. 112]{MR0482317}. Only bounded operator coefficients are considered in that paper but the operator matrix functions $E$ and $F$ in the equivalence relation \eqref{1eqdef} are independent of $P_0$. Hence they remain bounded also when $P_0$ is unbounded.
However, the method in \cite{MR0482317} can not be used directly if $P_i$ is unbounded for some $i>0$. The following example illustrates the problem for a quadratic polynomial. 

\begin{ex}
 Consider the operator polynomial $P:\C\rightarrow \Ls(\Hs)$ defined as
\[
	P(\lambda):=\lambda^2+\lambda A+B,\quad \dom (P(\lambda)) := \dom (A),\quad\lambda\in\C,
\]
where $A\in\Ls(\Hs)$ is an unbounded operator and $B\in\Bs(\Hs)$. Then the method in \cite{MR0482317}  is not applicable to find an equivalent linear problem after extension as $E(\lambda)$ and $E(\lambda)^{-1}$ would be unbounded for all $\lambda$ as can be seen below:
\[
	\begin{bmatrix}
	P(\lambda)\\
	& I_{\Hs}
	\end{bmatrix}=\begin{bmatrix}
	-I_{\Hs} &-A-\lambda\\
	 & I_{\Hs}
	\end{bmatrix}\begin{bmatrix}
	-A-\lambda & -B\\
	I_{\Hs} &-\lambda 
	\end{bmatrix}\begin{bmatrix}
	\lambda &I_{\Hs} \\
	I_{\Hs}&
	\end{bmatrix}.
\]
 However for all $\lambda\neq 0$, an equivalent spectral problem is $S(\lambda):=P(\lambda)/\lambda=A-\lambda-(-B)/(-\lambda)$. By extending $S(\lambda)$ by $-\lambda I_\Hs$ an equivalent problem is given by Lemma \ref{2Schur} as
\[
	\begin{bmatrix}
	S(\lambda)\\
	& -\lambda
	\end{bmatrix}=\begin{bmatrix}
	-I_{\Hs} &\frac{B}{\lambda}\\
	 & I_{\Hs}
	\end{bmatrix}\begin{bmatrix}
	-A-\lambda & -B\\
	I_{\Hs} &-\lambda 
	\end{bmatrix}\begin{bmatrix}
	I_{\Hs} &\\
	\frac{1}{\lambda}& I_{\Hs}
	\end{bmatrix},
\]
and as a consequence $P(\lambda)\oplus W(\lambda)=E(\lambda)(T-\lambda)F(\lambda)$ with $W(\lambda)=-\lambda$ and
\[
	E(\lambda)=\begin{bmatrix}
	-I_{\Hs} &\frac{B}{\lambda}\\
	 & I_{\Hs}
	\end{bmatrix}, \quad T=\begin{bmatrix}
	-A & -B\\
	I_{\Hs} &
	\end{bmatrix}, \quad F(\lambda)=\begin{bmatrix}
	\lambda &\\
	I_{\Hs}& I_{\Hs}
	\end{bmatrix}.
\]
Using this method, the obtained $T$ has the same entries as the operator given in \cite[p. 112]{MR0482317}, but the functions $E(\lambda)$, $F(\lambda)$ are bounded for $\lambda\neq 0$. Inspired by the previous example, we show how an equivalence can be found independent of which operator $P_i$ in Lemma \ref{2Schurpol} that is unbounded. Note that  Lemma \ref{2Schurpol} is the standard companion block linearization for operator polynomials formulated as an equivalence after extension.
\end{ex}
\begin{lem}\label{2Schurpol}
Let $P$ denote the operator polynomial defined in \eqref{2polop} and assume that $P_d$ is invertible. 
For $i<d$ set $\mon{P}_i:=P_d^{-1}P_i$ and  $\mon{P}_d:=I_{\Hs}$.  Let $\Cdom':=\C$ if $l=0$, and  $\Cdom':=\C\setminus\{0\}$ otherwise. Define the operator matrix $T\in\Ls(\Hs^d)$ on its natural domain as
\[
	T:=
	\begin{bmatrix}
	-\mon{P}_{d-1}&\cdots&-\mon{P}_1& -\mon{P}_0\\
	I_{\Hs}& & & \\
	& \ddots & & \\
	&  &I_{\Hs}&
	\end{bmatrix}.
\]
Further, define the operator matrix function $W:\Cdom' \rightarrow\Ls(\Hs^{\max(d-1,l)})$ as
\[
	W(\lambda):=\begin{bmatrix}
	I_{\Hs^{d-1-l}}\\
 	& -\lambda\\
 	& I_{\Hs}& \ddots\\
 	&& \ddots& \ddots\\
 	& & &I_{\Hs} & -\lambda
	\end{bmatrix},\quad \lambda\in {\Cdom}'.
\]
Then, the following equivalence results hold:
\begin{itemize}
\item[{\rm i)}] if $l<d$, $P(\lambda)\oplus W(\lambda)$ is equivalent to $T-\lambda$ for all $\lambda\in\Cdom'$.
\item[{\rm ii)}] if $l=d$,  $P(\lambda)\oplus W(\lambda)$ is equivalent to $P_d\oplus (T-\lambda)$ for all $\lambda\in\Cdom'$.
\end{itemize}

The operator matrix functions in the equivalence relation \eqref{1eqdef} are for $\lambda\in\Cdom'$ defined in the following steps: 
For $l<d$, define the operator matrix functions $E_\alpha,F_\alpha:\Cdom'\rightarrow\Ls(\Hs^{d-l})$ as
\[
\begin{array}{l}
	E_\alpha(\lambda):=
\begin{bmatrix}
-P_d &-\sum_{k=0}^{1}\lambda^{k}P_{d-1+k}&\hdots&\hdots&  -\sum_{k=0}^{d-l-1}\lambda^{k}P_{l+1+k} \\
&  I_{\Hs}& \lambda & \hdots&\lambda^{d-l-2}\\
& &\ddots & \ddots&\vdots\\
& & &\ddots&\lambda\\
 & & & & I_{\Hs}\\
\end{bmatrix},\\
	F_\alpha(\lambda):=\begin{bmatrix}
\lambda^{d-1}& I_{\Hs} & & \\
\vdots& & \ddots &  \\
\lambda^{l-1}& & &  I_{\Hs}  \\
\lambda^l& & \\
\end{bmatrix},
\end{array}
\]
whereas for $l=d-1$ define $E_\alpha(\lambda):=-P_d$ and $F_\alpha(\lambda):=\lambda^{d-1}I_\Hs$. 

For $l>0$ define  the operators matrix functions $E_\beta:\Cdom'\rightarrow\Bs(\Hs^{l}, \Hs^{\max(d-l,1)})$ and $F_\beta:\Cdom'\rightarrow\Bs( \Hs^{\max(d-l,1)},\Hs^{l})$  by
\[
E_\beta(\lambda):=
\begin{bmatrix}
 \sum_{k=0}^{l-1}\frac{P_k}{\lambda^{l-k}}&\hdots&\sum_{k=0}^{1}\frac{P_k}{\lambda^{2-k}}&\frac{P_0}{\lambda}\\
 0& \hdots&0& 0
\end{bmatrix},\quad F_\beta(\lambda):=
\begin{bmatrix}
 \lambda^{l-1}&0\\
 \vdots&\vdots\\
 I_\Hs&0
\end{bmatrix},
\]
where for $l\geq d-1$ we use the convention that the $0$-row/column vanish. If $l=d$, we define the operators  $E_\gamma\in\Bs(\Hs, \Hs^d)$ and $F_\gamma\in\Bs(\Hs^d, \Hs)$ as
\[
	E_\gamma:=\begin{bmatrix}
	P_d^{-1}\\
	0
	\end{bmatrix},\quad F_\gamma:=\begin{bmatrix}
	\mon{P}_{d-1}& \hdots& \mon{P}_0
	\end{bmatrix}.
\]
Then, for all $\lambda\in\Cdom'$ the operator matrix functions $E$ and $F$ in the equivalence relation \eqref{1eqdef} are given by 
\[
\begin{array}{l l l}
E(\lambda):=E_\alpha(\lambda),& F(\lambda):=F_\alpha(\lambda),& l=0,\\
E(\lambda):=\begin{bmatrix}
E_\alpha(\lambda)& E_\beta(\lambda)\\
& I_{\Hs^l}
\end{bmatrix},&
F(\lambda):=\begin{bmatrix}
F_\alpha(\lambda)& \\
F_\beta(\lambda)& I_{\Hs^l}
\end{bmatrix},& 0<l<d,\\
E(\lambda):=\begin{bmatrix}
\frac{P(\lambda)P_d^{-1}}{\lambda^d}& E_\beta(\lambda)\\
E_\gamma& I_{\Hs^d}
\end{bmatrix},&
F(\lambda):=\begin{bmatrix}
\sum_{i=0}^d\lambda^i\mon{P}_i& F_\gamma\\
F_\beta(\lambda)& I_{\Hs^d}
\end{bmatrix},& l=d.\\
\end{array}
\]

\end{lem}
\begin{proof}
For $l=0$, the result follows in principle from \cite[p. 112]{MR0482317}. Hence, we show the claim for $l>0$ and $\Cdom'=\C\setminus\{0\}$.
Define for all $\lambda\in\Cdom'$ the operator function $S$ by
\begin{equation*}\label{3conrat}
	S(\lambda):=\frac{P(\lambda)}{\lambda^l}=\sum_{k=0}^{d-l}\lambda^kP_{k+l}+\sum_{k=0}^{l-1}\frac{P_k}{\lambda^{l-k}},\quad \dom(R(\lambda))=\dom(P(\lambda)).
\end{equation*}
Assume $l<d$, then apart from  the sum $\sum_{k=0}^{l-1}P_k/\lambda^{l-k}$, $S$ is polynomial in $\lambda$ and only the zeroth-order term $P_l$ can be unbounded. Then, from \cite[p. 112]{MR0482317} it can be seen that $S$ is after $I_{\Hs^{d-1-l}}$-extension  equivalent to 
\[
\widehat{T}(\lambda):=
	\begin{bmatrix}
-\mon{P}^{-1}_{d}&\cdots&-\mon{P}_{l+1}& -\mon{P}_l-\sum_{k=0}^{l-1}\frac{\mon{P}_k}{\lambda^{l-k}}\\
I_{\Hs}& & & \\
& \ddots & & \\
& & I_{\Hs}
\end{bmatrix}.
\]
Since, the following identity holds, 
\[
	\sum_{k=0}^{l-1}\frac{\mon{P}_k}{\lambda^{l-k}}=
	-\begin{bmatrix}\mon{P}_{l-1}& \dots& \mon{P}_0
	\end{bmatrix}\begin{bmatrix}
	-\lambda\\
	I_{\Hs}&-\lambda\\
	&\ddots& \ddots\\
	&&I_{\Hs}&-\lambda \\
	\end{bmatrix}^{-1}\begin{bmatrix}I_{\Hs}\\
	 \\
	 \\
	\end{bmatrix},
\]
Theorem \ref{2schurrat0} gives that $S(\lambda)$ after $W(\lambda)$-extension is equivalent to $T-\lambda$ on $\Cdom$. By multiplying the first column in $S(\lambda)\oplus W(\lambda)$ with $\lambda^l$ the same result is obtained for $P(\lambda)$. The operators  $E(\lambda)$, $F(\lambda)$ are obtained by multiplying the corresponding operator matrix functions for the different equivalences.

 For $l=d$, Theorem \ref{2schurrat0} gives that $S(\lambda)\oplus W(\lambda)$ is equivalent  to
\[
	\widetilde{T}(\lambda):=
	\begin{bmatrix}
	P_d&P_{d-1}&P_{d-2}& \dots& P_0\\
	I_{\Hs}&-\lambda\\
	&I_{\Hs}&-\lambda\\
	&&\ddots& \ddots\\
	&&&I_{\Hs}&-\lambda \\
	\end{bmatrix}.
\]
Since $T-\lambda$ can be written in the form
\[
T-\lambda=
	\begin{bmatrix}
	-\lambda\\
	I_{\Hs}&-\lambda\\
	&\ddots& \ddots\\
	&&I_{\Hs}&-\lambda \\
	\end{bmatrix}-\begin{bmatrix}
	I_{\Hs}\\
	\\
	\\
	\\
	\end{bmatrix}P_d^{-1}\begin{bmatrix}
	P_{d-1}&P_{d-2}& \dots& P_0\\
	\end{bmatrix},
\]
it follows from Theorem \ref{2schurrat0} that $P_d\oplus (T-\lambda)$ is equivalent to $\widetilde{T}(\lambda)$. 
\end{proof}
\begin{ex}\label{ex:pencil}
In Lemma \ref{2Schurpol}, the result is rather different when $l=d$ even though $T$ has the same entries. In this case the equivalence is after both $P(\lambda)$ and $T-\lambda$ have been extended with an operator function and the following example shows that this extension  in general cannot be avoided. Let $A\in\Ls(\Hs),B\in\Bs(\Hs)$ and define $P:\C\setminus\{0\}\rightarrow\Ls(\Hs)$ as
\[
	P(\lambda):=\lambda A+B,\quad \dom(P)=\dom(A),
\]
where $A$ is invertible. If $A$ is bounded, $P(\lambda)$ is equivalent to $T-\lambda$, $T=-A^{-1}B$ but this equivalence do not hold if $A$ is unbounded. However, these operator functions are equivalent on $\C\setminus\{0\}$ after operator function extension as can be seen from Lemma \ref{2Schurpol} where the lemma  for $\lambda\in\C\setminus\{0\}$ gives that 
\[
	\begin{bmatrix}
	P(\lambda)\\
	&-\lambda
	\end{bmatrix}=\begin{bmatrix}
	I_\Hs+\frac{BA^{-1}}{\lambda}& \frac{B}{\lambda}\\
	A^{-1}&I_\Hs
	\end{bmatrix}\begin{bmatrix}
	A\\
	&T-\lambda
	\end{bmatrix}\begin{bmatrix}
	A^{-1}B+\lambda& A^{-1}B\\
	I_\Hs&I_\Hs
	\end{bmatrix}.
\] 
\end{ex}

\begin{thm}\label{2genP1lin}
Let $P$, $E$, $F$, and $W$ denote the operator functions on $\Cdom'\supset\Cdom$ defined in Lemma \ref{2Schurpol} and let $\mon{P}_i$, $i=1,\hdots, d$ denote the operators in that lemma. The operator matrix function $\om{P}:\Cdom\rightarrow\Ls(\Hs\oplus\HS,\Hs\oplus\HS')$ is on its natural domain  defined as 
\[
	\om{P}(\lambda):=\begin{bmatrix}
	P(\lambda) & X(\lambda)\\
	Q(\lambda) & Z(\lambda)
	\end{bmatrix}, \quad \lambda\in\Cdom,
\]
where
\[
	Q(\lambda)=\sum_{i=0}^{d-1}\lambda^iQ_i,\quad Q_i\in\Ls(\Hs,\HS'), \quad \lambda\in\Cdom.
\]
Assume that $Q_i\in\Bs(\Hs,\HS)$ for $i\neq l$ and if $l=d$ then $\overline{P_d^{-1}X(\lambda)}\in \Bs(\HS,\Hs)$ for all $\lambda\in\Omega$. 
Define for all $\lambda\in\Cdom$ the operator matrix function $\om{T}:\Cdom\rightarrow\Ls(\Hs^d\oplus\HS,\Hs^d\oplus\HS')$ on its natural domain as
\[
	\om{T}(\lambda):=
	\begin{bmatrix}
	-\mon{P}_{d-1}-\lambda& -\mon{P}_{d-2}&\cdots&-\mon{P}_1& -\mon{P}_0&-P^{-1}_{d}X(\lambda)\\
	I_\Hs& -\lambda & \\
	&I_\Hs&\ddots & \\
	& &\ddots & -\lambda& \\
	&  & &I_\Hs&-\lambda \\
	Q_{d-1}& Q_{d-2}&\cdots&Q_1& Q_0&Z(\lambda)
	\end{bmatrix}.
\]
Then, with respect to \eqref{3remlin}, the following equivalence results hold:
\begin{itemize}
	\item[{\rm i)}] if $l<d$, $\om{P}(\lambda)\oplus W(\lambda)$ is equivalent to $\om{T(\lambda)}$ for all $\lambda\in\Cdom$.
	\item[{\rm ii)}] if $l=d$,  $\om{P}(\lambda)\oplus W(\lambda)$ is equivalent to $P_d\oplus \om{T}(\lambda)$ for all $\lambda\in\Cdom$.
\end{itemize}

The operator matrix functions in the equivalence relation \eqref{1eqdef} are for $\lambda\in\Cdom$ defined in the following steps:

If $l<d$, define the operator matrix function $\side{E}_\alpha:\Cdom\rightarrow\Ls(\Hs^{d-l},\HS)$ as
\[
	\side{E}_\alpha(\lambda):=\begin{bmatrix}
	0 & -Q_{d-1}& -\sum_{k=0}^{1}\lambda^{k}Q_{d-2+k}&\cdots& -\sum_{k=0}^{d-l-2}\lambda^{k}Q_{l+1+k}
	\end{bmatrix},
\]
where $\side{E}_\alpha(\lambda):=0$ for $l=d-1$.

If $l>0$, define the operator matrix function $\side{E}_\beta:\Cdom\rightarrow\Bs(\Hs^{l}, \HS)$,
\[
\side{E}_\beta(\lambda):=\begin{bmatrix}
\sum_{k=0}^{l-1}\frac{Q_k}{\lambda^{l-k}}&\hdots&\sum_{k=0}^{1}\frac{Q_k}{\lambda^{2-k}}&\frac{Q_0}{\lambda}
\end{bmatrix}.
\]

The operator matrices $\side{E}:\Cdom\hspace{-.5pt}\rightarrow \hspace{-.5pt}\Bs(\Hs^{\max(d,l+1)},\HS)$ and $\side{F}:\Cdom\hspace{-.5pt}\rightarrow\hspace{-.5pt}\Bs(\HS,\Hs^{\max(d,l+1)})$ are then defined  as 
\begin{equation}\label{4polywide}
	\begin{array}{l l l}
	\side{E}(\lambda):=\side{E}_\alpha(\lambda),& \side{F}(\lambda):=0,& l=0,\\
	\side{E}(\lambda):=\begin{bmatrix}
	\side{E}_\alpha(\lambda)&\side{E}_\beta(\lambda)
	\end{bmatrix},&
	\side{F}(\lambda):=0,& 0<l<d,\\
	\side{E}(\lambda):=\begin{bmatrix}
	\frac{Q(\lambda)P_d^{-1}}{\lambda^d}&\side{E}_\beta(\lambda)
	\end{bmatrix},&
	\side{F}(\lambda):=\begin{bmatrix}
	\overline{P_d^{-1}X(\lambda)}\\
	0
	\end{bmatrix},& l=d.\\
	\end{array}
\end{equation}
Finally define the operator matrices $\om{E}(\lambda)$ and $\om{F}(\lambda)$ in the equivalence relation \eqref{1eqdef}:
\[
	\om{E}(\lambda):=\begin{bmatrix}
	E(\lambda)\\
	\side{E}(\lambda)& I_{\HS'}
	\end{bmatrix}, \quad
	\om{F}(\lambda):=\begin{bmatrix}
	F(\lambda)&\side{F}(\lambda)\\
	& I_{\HS}
	\end{bmatrix}.
\]
\end{thm}

\begin{proof}
Similar to the proof of Theorem \ref{2schurrat0}, where Lemma \ref{2Schurpol} with \eqref{4polywide} is used in Lemma \ref{2partblin}. Note that $\overline{P^{-1}_{d}X(\lambda)}=P^{-1}_{d}X(\lambda)$ on $\dom(X(\lambda))$.

\end{proof}
\begin{rem}
Theorem \ref{2genP1lin} requires $Q$ to be an operator polynomial. For a general $Q$ an equivalence is obtained by using the equivalence given in Lemma \ref{2Schurpol} together with Lemma \ref{2partblin} with $\side{E}:=0$ and $\side{F}:=0$.
\end{rem}

\section{Linearization of classes of operator matrix functions}
In Section \ref{Sec3} we considered three types of operator functions. One vital property differs between operator functions of the forms \eqref{3eq:rat} and \eqref{3polyop} compared to operator polynomials \eqref{2polop}: For polynomials the equivalence is to a linear operator function (Lemma \ref{2Schurpol}), but it is clear that a similar result will not hold in general for  \eqref{3eq:rat} and  \eqref{3polyop}.

If $A$, $B$, $C$, and $D$ in \eqref{3eq:rat} and $M_1,\hdots,M_n$ in \eqref{3polyop} are operator polynomials, Lemma \ref{2Schur} respective Lemma \ref{3jordig} can be used to find an equivalence after operator function extension to an operator matrix polynomial. Hence, if the entries in a $n\times n$ operator matrix function are either multiplications of polynomials or  Schur complements, then Theorem \ref{2schurrat0} and Theorem \ref{3jordigblo} can be used iteratively to find an equivalence to a operator matrix polynomial. An example of this form is considered in Section \ref{sec:ex}.

\subsection{Linearization of operator matrix polynomials}
Set $\Hs:=\oplus_{i=1}^n\Hs_i$ and consider the operator matrix polynomial $\om{P}:\C\rightarrow\Ls(\Hs)$, defined on it natural domain as
 \begin{equation}\label{2genP2}
\begin{array}{l}
\om{P}(\lambda):=\begin{bmatrix}
P_{1,1}(\lambda) &\hdots&P_{1,n}(\lambda) \\
\vdots & \ddots &\vdots\\
P_{n,1}(\lambda) & \hdots &P_{n,n}(\lambda) \\
\end{bmatrix},\quad \lambda\in\C,
\end{array}
\end{equation}
where $P_{j,i}(\lambda):=\sum_{k=0}^{d_{i,j}} \lambda^kP_{j,i}^{(k)}$ and $P_{j,i}^{(k)}\in \Ls(\Hs_i,\Hs_j)$. 
There are different ways to formulate \eqref{2genP2} that highlight different methods to linearize the operator matrix polynomial.
By using the notation: $P_{j,i}^{(k)}:=0$ for $k>d_{j,i}$ and $d:=\max d_{j,i}$, it follows that $\om{P}$ can be written in the form
 \begin{equation}\label{4larpol}
\om{P}(\lambda)=\sum_{k=0}^{d}\lambda^k\om{P}_k,\quad \om{P}_k:=
\begin{bmatrix}
P_{1,1}^{(k)}&\hdots&P_{1,n}^{(k)} \\
\vdots & \ddots &\vdots\\
P_{n,1}^{(k)}& \hdots &P_{n,n}^{(k)} \\
\end{bmatrix}.
\end{equation}
In the formulation \eqref{4larpol}, the problem is written as a single operator function, which makes it possible to utilize Lemma \ref{2Schurpol}, provided certain conditions hold. This is the most commonly used formulation, see e.g., \cite{MR1911850}. For the original formulation \eqref{2genP2}, Theorem \ref{2genP1lin} can be applied iteratively for each column, which results in a linear function. In Theorem \ref{polygenlin} we present the linearization obtained using this method and in Section \ref{sec:red} we will present a systematic approach to linearize operator matrix polynomials that relies on Theorem \ref{polygenlin}.

\begin{thm}\label{polygenlin}
Let $\om{P}$ be the  operator matrix polynomial \eqref{2genP2}, where $d_i:=d_{i,i}>0$ and $d_i>d_{j,i}$ for $j\neq i$. Assume that $P_{i,i}^{(d_{i})}$ are invertible and that there exist constants $l_i\in \{0,\dots,d_i\}$ such that $P_{j,i}^{(k)}\in \Bs(\Hs_i,\Hs_j)$ for $k\neq l_i$. For $k<d_i$ set $\mon{P}_{i,j}^{(k)}:={P_{i,i}^{(d_i)}}^{-1}P_{i,j}^{(k)}$ and $\mon{P}_{i,i}^{(d_i)}:=I_{\Hs_i}$. Let $\Cdom:=\C$ if $l_i=0$ for all $i$,  $\Cdom:=\C\setminus\{0\}$ otherwise. If $l_i=d_i$ assume that $\overline{\mon{P}_{i,j}^{(k)}}\in\Bs(\Hs_j,\Hs_i)$ for all indices $k$, $j$.  Define the operator matrix
 \[
	 \om{T}\in\Ls\left(\bigoplus_{i=1}^n\Hs_i^{d_{i}}\right)\quad
 	\text{as}\quad 
	\om{T}:=\begin{bmatrix}
	T_{1,1}&\hdots &T_{1,n}\\
	\vdots&\ddots &\vdots\\
	T_{n,1}&\hdots &T_{n,n}\\
	\end{bmatrix},
\]
where $T_{j,i}\in\Ls(\Hs_i^{d_i},\Hs_j^{d_j})$ are the operator matrices
\[
	T_{j,i}:=\left\{
	\begin{array}{l l}
	\begin{bmatrix}
	-\mon{P}_{i,i}^{(d_i-1)}& \cdots&-\mon{P}_{i,i}^{(1)}& -P_{i,i}^{(0)}\\
	I_{\Hs_i}& & & \\
	& \ddots & & \\
	&  &I_{\Hs_i}&
	\end{bmatrix},& i=j,
	\vspace{5pt}\\
	\begin{bmatrix}
	-\mon{P}_{j,i}^{(d_i-1)}&\cdots&-\mon{P}_{j,i}^{(1)}& -\mon{P}_{j,i}^{(0)}\\
	0&\hdots&0&0\\
	\end{bmatrix},& i\neq j.
	\end{array}
	\right.
\]
Let $\om{W}(\lambda):=\oplus_{i=1}^n W_i(\lambda)$,  where $W_i:\Cdom\rightarrow\Ls(\Hs_i^{\max(d_i-1,l_i)})$ are the operator matrix functions
\[
	\begin{array}{l}
	 W_i(\lambda)
	 :=\begin{bmatrix}
	I_{\Hs_i^{d_i-l_i-1}}\\
	 & -\lambda\\
	 &  I_{\Hs_i}& \ddots\\
	 & & \ddots& \ddots\\
	 & & &I_{\Hs_i} & -\lambda
	\end{bmatrix},\quad \lambda\in\Cdom.
	\end{array}
\]
Set $L:=\{i\in\{1,\dots,n\}: l_i=d_i\}$. Then the following results hold:
\begin{itemize}
\item[{\rm i)}] if $L=\emptyset$, $\om{P}(\lambda)\oplus\om{W}(\lambda)$ is equivalent to $\om{T}-\lambda$ for all $\lambda\in\Cdom$.
\item[{\rm ii)}] if $L\neq\emptyset$, $\om{P}(\lambda)\oplus\om{W}(\lambda)$ is equivalent to $\om{P}_d\oplus(\om{T}-\lambda)$ for all $\lambda\in\Cdom$, where
\[
	\om{P}_d:=\bigoplus_{i\in L} P_{i,i}^{(d_{i})}\in \Ls\left(\bigoplus_{i\in L}\Hs_i\right)
\]
is defined on its natural domain.
\end{itemize}

In the case $L=\emptyset$ the operator matrix functions in the equivalence relation \eqref{1eqdef}  with respect to the structure \eqref{3remlin} are defined in the following steps:
Let the operator matrix functions $E_{i}^{(\alpha)},F_{i}^{(\alpha)}:\Cdom\rightarrow\Bs(\Hs_i^{d_i-l_i})$ and $\side{E}_{j,i}^{(\alpha)}:\Cdom\rightarrow \Bs(\Hs_i^{d_i-l_i},\Hs_j^{d_j})$ for $i\neq j$ be defined as 
\[
\begin{array}{l}
	E_{i}^{(\alpha)}(\lambda):=
\begin{bmatrix}
-P_{i,i}^{(d_i)} &-\sum_{k=0}^{1}\lambda^{k}P_{i,i}^{(d_i-1+k)}&\hdots& - \sum_{k=0}^{d_i-l_i-1}\lambda^{k} P_{i,i}^{(l_i+1+k)}\\
&  I_{\Hs_i} & \hdots&\lambda^{d_i-l_i-2}\\
& &\ddots &\vdots\\
 & &  & I_{\Hs_i}\\
\end{bmatrix},\\
	F_{i}^{(\alpha)}(\lambda):=\begin{bmatrix}
\lambda^{d_i-1}& I_{\Hs_i} &  &  \\
\vdots& &\ddots &  \\
\lambda^{l_i-1}& & & I_{\Hs_i}\\
\lambda^{l_i}& &  \\
\end{bmatrix},\vspace{2pt}\\ 
{E}_{j,i}^{(\alpha)}(\lambda):=
\begin{bmatrix}
0&  -\sum_{k=0}^{0}\lambda^{k}P_{j,i}^{(d_i-1+k)}&
\cdots& -\sum_{k=0}^{d_i-l_i-2}\lambda^{k}P_{j,i}^{(l_i+1+k)}\\
0&0&\dots&0
\end{bmatrix}.
\end{array}
\]
Note, if $l_i=d_i-1$ this means that $E_{i}^{(\alpha)}(\lambda):=-P_{i,i}^{(d_i)}$, $F_{i}^{(\alpha)}(\lambda):=\lambda^{d_i-1}$ and $E_{j,i}^{(\alpha)}(\lambda):=0$.
If $l_i>0$, define for $i\neq j$ the operator matrix functions $E_{i}^{(\beta)}:\Cdom\rightarrow\Bs(\Hs_i^{l_i}, \Hs_i^{d_i-l_i})$, $F_{i}^{(\beta)}:\Cdom\rightarrow\Bs(\Hs_i^{d_i-l_i}, \Hs_i^{l_i})$, and $E_{j,i}^{(\beta)}:\Cdom\rightarrow \Bs(\Hs_i^{l_i},\Hs_j^{d_j})$ as
\[
	\begin{array}{l l}
	E_{i}^{(\beta)}(\lambda):=
	\begin{bmatrix}
	 \sum_{k=0}^{l_i-1}\frac{P_{i,i}^{(k)}}{\lambda^{l_i-k}}&\hdots&\sum_{k=0}^{1}\frac{P_{i,i}^{(k)}}{\lambda^{2-k}}&\frac{P_{i,i}^{(0)}}{\lambda}\\
	 0& \hdots& 0&0\\
	\end{bmatrix},& F_{i}^{(\beta)}(\lambda):=
	\begin{bmatrix}
	 \lambda^{l_i-1}&0\\
	 \vdots&\vdots\\
	 I_{\Hs_i}&0
	\end{bmatrix},\\
	E_{j,i}^{(\beta)}(\lambda):=\begin{bmatrix}
	\sum_{k=0}^{l_i-1}\frac{P_{j,i}^{(k)}}{\lambda^{l_i-k}}&\hdots&\sum_{k=0}^{1}\frac{P_{j,i}^{(k)}}{\lambda^{2-k}}&\frac{P_{j,i}^{(0)}}{\lambda}\\
	 0& \hdots& 0&0\\
	\end{bmatrix}.
	\end{array}
\]

For $i\neq j$ define the operators matrices:
\[
	\begin{array}{l l l}
	E_{i,i}(\lambda)=E_{i}^{(\alpha)}(\lambda),& F_{i}(\lambda)=F_{i}^{(\alpha)}(\lambda),& l_i=0,\\
	E_{i,i}(\lambda)=\begin{bmatrix}
	E_{i}^{(\alpha)}(\lambda)& E_{i}^{(\beta)}(\lambda)\\
	& I_{\Hs_i^{l_i}}
	\end{bmatrix},&
	F_{i}(\lambda)=\begin{bmatrix}
	F_{i}^{(\alpha)}(\lambda)& \\
	F_{i}^{(\beta)}(\lambda)& I_{\Hs_i^{l_i}}
	\end{bmatrix},& l_i>0,\\
	\end{array}
\]

\[
	\begin{array}{l l l}
	E_{j,i}(\lambda)=E_{j,i}^{(\alpha)}(\lambda),& l_i=0,\\
	E_{j,i}(\lambda)=\begin{bmatrix}
	E_{j,i}^{(\alpha)}(\lambda)&E_{j,i}^{(\beta)}(\lambda)
	\end{bmatrix},& l_i>0.
	\end{array}
\]
Then the operator matrices $\mathcal{E}(\lambda)$ and $\mathcal{F}(\lambda)$  in the equivalence relation \eqref{1eqdef} are
\[	
	\mathcal{E}(\lambda)=\begin{bmatrix}
	E_{1,1}(\lambda)&\hdots &E_{1,n}(\lambda)\\
	\vdots&\ddots &\vdots\\
	E_{n,1}(\lambda)&\hdots &E_{n,n}(\lambda)\\
	\end{bmatrix},\quad 
	\mathcal{F}(\lambda)=\begin{bmatrix}
	F_{1}(\lambda)\\
	 &\ddots &\\
	&&F_{n}(\lambda)\\
	\end{bmatrix}.
\]

\end{thm}
\begin{proof}
The claims follows from applying Theorem \ref{2genP1lin} to each column in \eqref{2genP2}. However, for columns $2,\hdots,n$  reordering of the diagonal blocks as in  \eqref{2genT2} is needed to be able to apply  Theorem \ref{2genP1lin} directly.
\end{proof}
\begin{rem}
In Theorem \ref{polygenlin} the operator matrix functions $\mathcal{E}$ and $\mathcal{F}$ in the equivalence relation \eqref{1eqdef} are not specified for the case $l_i=d_i$. The reason is that then $\om{E}(\lambda)$ and $\om{F}(\lambda)$ depend on the order of which Theorem \ref{2genP1lin} is applied to the columns and  are very complicated albeit possible to determine.
\end{rem}

\begin{rem}
For operator polynomials it is common to consider equivalence after extension to a non-monic linear operator pencil, $\om{T}-\lambda\om{S}$, \cite{MR0482317}. In Theorem \ref{polygenlin} the condition that $P_{i,i}$ is invertible for $i=1,\hdots,n$ can be dropped if the matrix block in the equivalence is non-monic. However, the reduction of a non-monic pencil to an operator is as pointed out by Kato \cite[VII, Section 6.1]{MR1335452} non-trivial; see also Example \ref{ex:pencil}.
\end{rem}

There are both advantages and disadvantages of using Theorem \ref{polygenlin} instead of Lemma \ref{2Schurpol} for operator matrix polynomials. One advantage is that $\om{P}_{d}$ does not have to be invertible. Furthermore, for unbounded operators functions Theorem \ref{polygenlin} can handle more cases since  it allows  $l_i\neq l_j$ while in Lemma \ref{2Schurpol}, $P_l$ is unbounded for at most one $l\in\{0,\hdots,d\}$.  However, a disadvantage of this method is that the highest degree  in each column has to be in the diagonal.  Importantly, if both methods are applicable for $\om{P}$,  then the obtained linearization using Theorem \ref{polygenlin} and Lemma \ref{2Schurpol} is the same up to ordering of the spaces. 
Even if the conditions on $\om{P}$ in Lemma \ref{2Schurpol}  and/or Theorem \ref{polygenlin}  are not satisfied an equivalent operator matrix function $\widehat{\om{P}}$ that satisfies these conditions can in many cases still be found. For example, Lemma \ref{2Schurpol} cannot be applied if the highest degree in the columns, $d_i$, are not the same. However, for $\lambda\in\Cdom\setminus\{0\}$ an equivalent operator matrix function is obtained as 
\[
	\widehat{\om{P}}(\lambda):=\om{P}(\lambda)
	\begin{bmatrix}
	\lambda^{d-d_1}\\
	& \ddots\\
	& & \lambda^{d-d_n}
	\end{bmatrix}, \quad \lambda\in\Cdom,
\]
where in $\widehat{\om{P}}$, the highest degree is the same in each column, unless one column is identically $0$. However, the coefficient to the highest order, $\widehat{\om{P}}_d$, might still be non-invertible and the boundedness condition might not be satisfied. Even if all conditions are satisfied the method increases the size of the linearization and introduces false solutions at $0$.  This is connected to the \emph{column reduction} concept for matrix polynomials discussed for example in \cite{MR1223472}. Due to these common problems that restrict use of  Lemma \ref{2Schurpol}  and the problems that can occur when trying  to find a suitable equivalent problem, we 
prefer to use the results in Theorem \ref{polygenlin}. Therefore we develop a method that for a given operator matrix polynomial $\om{P}$ provides an equivalent operator matrix polynomial $\widehat{\om{P}}$ for which the conditions in Theorem \ref{polygenlin} are satisfied. 

\subsection{Column reduction of operator matrix polynomials}\label{sec:red}
Theorem \ref{polygenlin} is only applicable when the diagonal entries in \eqref{2genP2} are of strictly higher degree than the degrees of the rest of the entries in the same column. 
The aim of this subsection is to find for given operator matrix polynomial $\om{P}$ a sequence of transformations that yields an equivalent operator matrix polynomial, where the diagonal entries have the highest degrees.

One type of column reduction algorithms of polynomial matrices was considered in \cite{MR1223472}, but the column reduction algorithms presented in this section are different also in the finite dimensional case. Naturally, new challenges emerge in the infinite dimensional case and when some of the operators are unbounded. This can be seen in the following example, which also illustrates that it is not necessary to have an equivalence in each step.

\begin{ex}\label{3redex}
Consider the operator matrix function $\om{P}:\C\rightarrow \Ls(\Hs_1\oplus\Hs_2\oplus\Hs_3)$
\[
	\om{P}(\lambda):=
	\begin{bmatrix}
	\lambda A&B&\lambda C\\
	\lambda D+\widehat{D} &\lambda G &\lambda^2 H+\widehat{H} \\
	J &&\lambda L
	\end{bmatrix},\quad \lambda\in\C,
\]
on its natural domain. $\om{P}$ does not have the highest degrees in the diagonal entries. However, under the assumptions stated at the end of the example, an equivalent operator matrix polynomial can be found, where the highest degrees are on the diagonal. In the following, we will apply particular transformations that for the general case are defined in \eqref{3redk}. Let $\widetilde{\mathcal{K}}_1$ denote the operator matrix
\[
	\widetilde{\mathcal{K}}_1:=\begin{bmatrix}
	I_{\Hs_1} & &\\
	-DA^{-1}& I_{\Hs_2}\\
	& &I _{\Hs_3}
	\end{bmatrix}.
\]	
The operator matrix function $\widetilde{\mathcal{K}}_1\om{P}$ is then
\[
	\widetilde{\mathcal{K}}_1\om{P}(\lambda)=
	\begin{bmatrix}
	\lambda A&B&\lambda C\\
	\widehat{D} &\lambda G -DA^{-1}B&\lambda^2 H-\lambda DA^{-1}C + \widehat{H}\\
	J & &\lambda L
	\end{bmatrix},\quad \lambda\in\C,
\]
which for the first two columns has the highest degree in the diagonal but not in the last column. Let $\widetilde{\mathcal{K}}_3$ denote the operator matrix function defined by
\[
	\widetilde{\mathcal{K}}_3(\lambda):=\begin{bmatrix}
	I_{\Hs_1} & &-CL^{-1}\\
	& I _{\Hs_2}&-(\lambda H-DA^{-1}C)L^{-1}\\
	& &I_{\Hs_3}
	\end{bmatrix},\quad \lambda\in\C.
\]
Then 
\begin{equation}\label{redegmessCE}
	\widetilde{\mathcal{K}}_3(\lambda)\widetilde{\mathcal{K}}_1\om{P}(\lambda)=
	\begin{bmatrix}
	\lambda A-CL^{-1}J&B&\\
	-\lambda HL^{-1}J+\widehat{D}+DA^{-1}CL^{-1}J &\lambda G -DA^{-1}B&\widehat{H} \\
	J & &\lambda L
	\end{bmatrix}.
\end{equation}
Hence, for $\widetilde{\mathcal{K}}_3\widetilde{\mathcal{K}}_1\om{P}$ the third column has the highest degree in the diagonal. However, in the first column the entry in the diagonal is not of strictly higher degree than the rest of the column. We will therefore apply the operator matrix
\[
	\widehat{\mathcal{K}}_1:=\begin{bmatrix}
	I_{\Hs_1} & &\\
	HL^{-1}JA^{-1}& I_{\Hs_2}\\
	& &I_{\Hs_3}
	\end{bmatrix}
\]
to \eqref{redegmessCE}. In order to justify the formal steps above, we first state some conditions on $\om{P}$. 
Assume that $A$, $L$ are invertible and $CL^{-1}$, $\overline{(D-HL^{-1}J)}A^{-1}$, $HL^{-1}$ are bounded. The domain of $\om{P}$ is chosen as
\[
	\dom(\om{P}):=(\dom(A)\cap\dom(\widehat{D})\cap\dom(J))\oplus(\dom(B)\cap\dom(G))\oplus(\dom(\widehat{F})\cap\dom(L)).
	\]
Let $E:\C\rightarrow\Bs(\Hs_1,\Hs_2,\Hs_3)$ be defined as $E(\lambda):=\overline{\widehat{\mathcal{K}}_1\widetilde{\mathcal{K}}_3(\lambda)\widetilde{\mathcal{K}}_1}$, where
\[
E(\lambda)=
	\begin{bmatrix}
		I_{\Hs_1} & & -CL^{-1}\\
		-\overline{(D-HL^{-1}J)}A^{-1}&I_{\Hs_2}&-\lambda HL^{-1}+\overline{(D-HL^{-1}J)}A^{-1}CL^{-1}\\
		& & I_{\Hs_3}
	\end{bmatrix}.
\]
Define $\widehat{\om{P}}:\C\rightarrow\Ls(\Hs_1,\Hs_2,\Hs_3)$, $\dom(\widehat{\om{P}})=\dom(\om{P})$ as $\widehat{\om{P}}(\lambda):=E(\lambda)\om{P}(\lambda)$, where
\[\label{reddegex2}
	\widehat{\om{P}}(\lambda)=
\begin{bmatrix}
\lambda A-CL^{-1}J&B&\\
\widehat{D}+\overline{(D-HL^{-1}J)}A^{-1}CL^{-1}J &\lambda G -\overline{(D-HL^{-1}J)}A^{-1}B&\widehat{H} \\
J & &\lambda L
\end{bmatrix}.
\]
The operator matrix polynomial $\widehat{\om{P}}$ has the highest degrees in the diagonal. Furthermore, since $E(\lambda)$ is bounded and  invertible for $\lambda\in\C$ it follows that  $\om{P}$ and $\widehat{\om{P}}$ are equivalent on $\C$.
\end{ex}

Example \ref{3redex} indicates that in the general case it is not feasible to obtain a closed formula for the final equivalent operator matrix polynomial. However, algorithms that follow the steps in Example \ref{3redex} will below be developed for bounded operator matrix polynomials. These algorithms also work for classes of operator matrix functions with unbounded entries, as in Example \ref{3redex}, and it is in each case possible to check if one of the algorithms is applicable.

Let $\om{P}$ denote the operator matrix polynomial \eqref{2genP2} and assume that for $i\neq j$ there exists operator polynomials $K_{j,i}(\om{P})$ and $R_{j,i}(\om{P})$ such that $P_{j,i}=K_{j,i}(\om{P})P_{i,i}+R_{j,i}(\om{P})$, where $\deg R_{j,i}(\om{P})<\deg P_{i,i}(\om{P})$. A sufficient condition for the existence of these operators is that $P_{i,i}^{(d_{i,i})}$ is invertible.

The dependence on $\om{P}:\C\rightarrow\Bs(\Hs)$ is written out explicitly since we want to use $K_{j,i}(\om{P}):\C\rightarrow \Bs(\Hs_{i},\Hs_j)$ in the algorithms. Define $\om{K}_{j,i}(\om{P})
 :\C\rightarrow \Bs(\Hs)$ as
\begin{equation}\label{3redk}
	\mathcal{K}_{j,i}(\om{P}):=\left\{
	\begin{array}{c l}
	\begin{bmatrix}
	I_{\Hs_1} \\
	 &\ddots & \\
	& -K_{j,i}(\om{P})&\ddots\\
	  & &&I_{\Hs_n} \\
	\end{bmatrix},&  i\neq j\text{ ($K_{j,i}$ is in position $(j,i)$)},\vspace{3 pt}\\
	I_\Hs, & i=j.
	\end{array}\right. 
\end{equation}
Multiplying an operator matrix polynomial $\om{P}$ from the left with $\mathcal{K}_{j,i}(\om{P})$ will be called \emph{reduction of the $i$-th column in the $j$-th row}. Additionally a column in $\om{P}$ is said to be \emph{reduced} if the highest degree is in the diagonal of $\om{P}$ in that column. When we in the algorithms presented below reduce the $(i,j)$-entry in $\om{P}$ the condition that $P_{j,i}=K_{j,i}(\om{P})P_{i,i}+R_{j,i}(\om{P})$ has a solution with $\deg R_{j,i}(\om{P})<\deg P_{i,i}(\om{P})$ is not stated explicitly. Moreover, the notation $\mathcal{K}_{l:k,i}(\om{P}):=\mathcal{K}_{l,i}(\om{P})\hdots\mathcal{K}_{k,i}(\om{P})$ is used and it is clear that $\mathcal{K}_{j,i}(\om{P})$ commutes so $\mathcal{K}_{l:k,i}(\om{P})$ is independent of the ordering in the multiplication. For convenience, 
 the notation $\mathcal{K}_{i}(\om{P}) :=\mathcal{K}_{1:n,i}(\om{P})$ is used. For example, the first column in the operator function $\widehat{\om{P}}$ defined by
\begin{equation}\label{2Ai1}
	\widehat{\om{P}}:=
	\mathcal{K}_1(\om{P}) \om{P}=\begin{bmatrix}
	P_{1,1}&P_{1,2}&\hdots&P_{1,n} \\
	R_{2,1}(\om{P})& \widehat{P}_{2,2}&\hdots&\widehat{P}_{2,n} \\
	\vdots &\vdots& \ddots &\vdots\\
	R_{n,1}(\om{P})& \widehat{P}_{n,2}&\hdots&\widehat{P}_{n,n}\\
	\end{bmatrix},
\end{equation}
is reduced. The entries in $\widehat{\om{P}}$ satisfy the conditions $\deg P_{1,1}>\deg R_{j,1}(\om{P})$ and $\widehat{P}_{j,i}:=P_{j,i}-K_{j,1}(\om{P})P_{1,i}$. 

With the notation above the operator functions defined in Example \ref{3redex} reads $E:=\overline{(\mathcal{K}_1\circ\mathcal{K}_3\circ\mathcal{K}_1)(\om{P})}$ and $\widehat{\om{P}}:=\overline{(\mathcal{K}_1\circ\mathcal{K}_3\circ\mathcal{K}_1)(\om{P})}\om{P}$.

\begin{defn}\label{def:matrix}
Let  $\om{P}:\C\rightarrow\Ls\left(\bigoplus_{i=1}^n \Hs_i\right)$ denote an operator matrix function with the operator polynomial entries $P_{j,i}:\C\rightarrow\Ls\left( \Hs_i,\Hs_j\right)$ and define
its $\R^{n\times n}$ \emph{degree matrix}
\[
	D(\om{P})=\begin{bmatrix}
	d_{1,1}&\hdots&d_{1,n} \\
	\vdots & \ddots &\vdots\\
	d_{n,1} & \hdots &d_{n,n}
	\end{bmatrix},
\]
where the $(i,j)$-th entry is the degree of $P_{i,j}$ and we set $d_{i,j}=-\infty$ if $P_{i,j}=0$. For given $D(\om{P})$ we define the \emph{difference matrix} 
\[
	\Delta (\om{P}):=\begin{bmatrix}
	d_{1,1}&\hdots&d_{1,n} \\
	\vdots & \ddots &\vdots\\
	d_{n,1} & \hdots &d_{n,n}
	\end{bmatrix}-\begin{bmatrix}
	d_{1,1}&\hdots&d_{n,n} \\
	\vdots & \ddots &\vdots\\
	d_{1,1} & \hdots &d_{n,n}
\end{bmatrix}.
\] 
\end{defn}
Define the functions
\begin{equation}\label{2supermax}
	f(x,y,z)=\bigg\{\begin{array}{l l}
	\max(x,y+z)& y\geq0\\
	x& y<0		
\end{array},
\end{equation}

and
\begin{equation}\label{3maxdiff}
	f_0(x,y,z,w)=f(x,y,z)-f(0,w,z).
\end{equation}
\begin{lem}\label{3funred}
The following properties hold for \eqref{3maxdiff}:
\begin{itemize}
	\item[{\rm i)}] $f_0(x,y,z,w)\leq\max(x,y+z)$.
	\item[{\rm ii)}] $f_0$ is non-decreasing in the first and second argument.
\end{itemize}
\end{lem}
\begin{proof}
{\rm i)} Follows from the inequalities $f(0,w,z)\geq0$ and $f(x,y,z)\leq \max(x,y+z)$.
\hspace{6pt}{\rm ii.)}  The function $f(x,y,z)$ is non-decreasing in $x$ and $y$, which implies the same properties for $f_0$.
\end{proof}

The case $\deg \widehat{P}_{j,i}<\max \{\deg P_{j,i},\deg K_{j,1}(\om{P})P_{1,i}\}$ in \eqref{2Ai1} can only occur if $\deg P_{j,i}=\deg K_{j,1}(\om{P})P_{1,i}$ and even then it is improbable in general. Therefore, in the following we assume that $\deg \widehat{P}_{j,i}=\max \{\deg P_{j,i},\deg K_{j,1}(\om{P})P_{1,i}\}$. This means that the degree matrix of $\widehat{\om{P}}$ is
\[
	D(\widehat{\om{P}})=\begin{bmatrix}
	d_{1,1}&d_{1,2}&\hdots &d_{1,n}\\
	m_{(d_{2,1},d_{1,1}-1)}&f(d_{2,2},\delt_{2,1},d_{1,2})&\hdots &f(d_{2,n},\delt_{2,1},d_{1,n})\\
	\vdots&\vdots&\ddots &\vdots\\
	m_{(d_{n,1},d_{1,1}-1)}&f(d_{n,2},\delt_{n,1},d_{1,2})&\hdots &f(d_{n,n},\delt_{n,1},d_{1,n})
	\end{bmatrix}, 
\] 
where $f$ is defined in \eqref{2supermax} and $\delt_{j,i}:=\Delta (\om{P})_{j,i}=d_{j,i}-d_{i,i}$ denote the matrix entries in Definition \ref{def:matrix}. Moreover, $m_{(x,y)}$ denotes a value that is less than or equal to $\min(x,y)$. It then follows that the difference matrix of $\widehat{\om{P}}$ is 
\[
	\Delta (\widehat{\om{P}})\hspace{-0pt}=\hspace{-0pt}\begin{bmatrix}
	\delt_{1,1}&f_0(\delt_{1,2},\delt_{1,1},\delt_{1,2},\delt_{2,1})&\hdots &f_0(\delt_{1,n},\delt_{1,1},\delt_{1,n},\delt_{n,1})\hspace{-0pt}\\
	m_{\delt_{2,1},-1}&f_0(\delt_{2,2},\delt_{2,1},\delt_{1,2},\delt_{2,1})&\hdots &f_0(\delt_{2,n},\delt_{2,1},\delt_{1,n},\delt_{n,1})\hspace{-0pt}\\
	\vdots&\vdots&\ddots &\vdots\\
	m_{\delt_{n,1},-1}&f_0(\delt_{n,2},\delt_{n,1},\delt_{1,2},\delt_{2,1})&\hdots &f_0(\delt_{n,n},\delt_{n,1},\delt_{1,n},\delt_{n,1})\hspace{-0pt}
	\end{bmatrix}\, ,
\]
where $f_0$ is given by \eqref{3maxdiff}. Hence, the difference matrix, $\Delta(\mathcal{K}_i(\om{P}) \om{P})$, can be computed using only the difference matrix $\Delta (\om{P})$, apart from the column $i$ where an upper estimate is found. This knowledge of the difference matrix is sufficient for the presented algorithms.

\begin{lem}\label{2rowlem}
Let $\om{P}$ be the  operator matrix polynomial \eqref{2genP2}. Assume $\Delta (\om{P})_{j,i}<0$ for all $j,i\leq k-1$ with $j\neq i$ and  $\Delta (\om{P})_{k,i}\leq\delta$  for $i\leq k-1$. Define the operator matrix polynomial $\widehat{P}:=E\om{P}$ where
\[
 	E=(\mathcal{K}_{k,k-1}\circ \hdots\circ\mathcal{K}_{k,1})^{\delta+1}(\om{P}).
\]
Then $\Delta (\widehat{\om{P}})_{j,i}<0$  for $j\neq i$ and $i\leq k-1$, $j\leq k$.
\end{lem}
\begin{proof}
Since $\Delta(\mathcal{K}_{k,1}(\om{P})\om{P})_{k,1}<0$ it follows from the definition of $f_0$ that \\ $\Delta(\mathcal{K}_{k,1}(\om{P})\om{P})_{k,i}\leq\delta$ for $2\leq i\leq k-1$. Hence, $\Delta((\mathcal{K}_{k,2}\circ\mathcal{K}_{k,1})(\om{P})\om{P})_{k,1}\leq\delta-1$, $\Delta((\mathcal{K}_{k,2}\circ\mathcal{K}_{k,1})(\om{P})\om{P})_{k,1}<0$, and $\Delta((\mathcal{K}_{k,2}\circ\mathcal{K}_{k,1})(\om{P})\om{P})_{k,i}\leq\delta$ for $3\leq i\leq k-1$. This implies $\Delta((\mathcal{K}_{k,k-1}\circ\hdots\circ\mathcal{K}_{k,1})(\om{P})\om{P})_{k,i}\leq\delta-1$ for $1\leq i\leq k-1$ and the result follows by induction.
\end{proof}

\begin{lem}\label{2diaglem}
Let $\om{P}$ be the  operator matrix polynomial \eqref{2genP2}. Assume that $\Delta(\om{P})_{j,i}\hspace{-1pt}<0$ for $k\geq i,j$ and $j\neq i>1$. Moreover, assume $\Delta(\om{P})_{j,1}\leq\Delta(\om{P})_{l,1}$ for $1<j<l\leq k$. Set $\delta:=\Delta(\om{P})_{k,1}$ and define  $\widehat{\om{P}}=E\om{P}$, where 
\[
	E:=\left\{\begin{array}{l l}
	\mathcal{K}_{2:k,1}(\om{P}),&\delta=0,\\
	\left(\mathcal{K}_{1:k,k-1}\circ\hdots\circ\mathcal{K}_{1:k,1}\circ\left(\mathcal{K}_{k:k,k-1}\circ\hdots\circ\mathcal{K}_{2:k,1}\right)^{\delta-1}\right)(\om{P}),&\delta>0.
	\end{array}
	\right. 
\] 
Then $\Delta(\widehat{\om{P}})_{j,i}<0$ for $i,j\leq k$ and $j\neq i$. 
 \end{lem}

\begin{proof}
If $\delta=0$ the result is trivial. Now let $\delta>0$ and define for $p\in\{0,\hdots,\delta-2\}$ and $q\in\{1,\hdots,k-1\}$ the operator 
\[
	\om{P}_p^{q}:=\left(\mathcal{K}_{q+1:k,q}\circ\hdots\circ\mathcal{K}_{2:k,1}\circ\left(\mathcal{K}_{k:k,k-1}\circ \hdots\circ\mathcal{K}_{2:k,1}\right)^{p}\right)(\om{P})\om{P}
\]
 and the constants $\delta_j=\Delta(\om{P})_{j,1}-\Delta(\om{P})_{j-1,1}$, for $j=2,\hdots,k$.

The non-negative values in the first $k$ columns of $\Delta(\om{P})$  are nondecreasing in the first $k$ rows. By Lemma \ref{3funred} {\rm ii}) $f_0$ is non-decreasing in the first and second argument. Thus, the non-negative values in the first $k$ columns of $\Delta(\om{P}_p^q)$ are nondecreasing in the first $k$ rows. This also implies that there can be no positive value above the diagonal in $\Delta(\om{P}_p^{q})$.
 
 The rest of the proof relies on showing that the following conditions hold
 \begin{equation}\label{cond1}
  	\Delta(\om{P}^q_{p})_{j,i}\leq \max(\Delta(\om{P}^q_p)_{j-1,i}+\delta_j,\delta_j-1,-1),\quad   \text{ for }k\geq j> i,
  \end{equation}
\begin{equation}\label{cond2}
	\begin{array}{l l l}
	\Delta(\om{P}^q_{p})_{j,i}\leq \max(\Delta(\om{P})_{j,1}-(p+2),-1)& q\geq i, & j>i,\\
	\Delta(\om{P}^q_{p})_{j,i}\leq \max(\Delta(\om{P})_{j,1}-(p+1),-1)& q<i,& j>i.
	\end{array}
\end{equation}
The proof of these conditions is based on induction over $p$ and $q$ and it is clear from the definition of $f_0$ that \eqref{cond1} and \eqref{cond2} are satisfied for $\om{P}_0^1$. 

 For $i=q+1$ the conditions \eqref{cond1} and \eqref{cond2} are satisfied trivially for $\Delta(\om{P}^{q+1}_{p})_{j,i}$. 
 Further 
 for $j<q+2$ the induction is trivial for both  \eqref{cond1}  and  \eqref{cond2}. Hence, in the following we assume $j\geq q+2$ and $i\neq q+1$.
 Let $\Delta(\om{P}^q_{p})$ satisfy the conditions \eqref{cond1}, \eqref{cond2} and take $q<k-1$. Then since $\Delta(\om{P}_p^{q+1})_{j,i}=\Delta(\mathcal{K}_{q+2:k,q+1}(\om{P}_p^q)\om{P}_p^q)_{j,i}$, we have 
 \[
	\Delta(\om{P}_p^{q+1})_{j,i}=f_0(\Delta(\om{P}_p^q)_{j,i},\Delta(\om{P}_p^q)_{j,q+1},\Delta(\om{P}_p^q)_{q+1,i},\Delta(\om{P}_p^q)_{i,q+1}).
 \]
 First we will show that condition \eqref{cond1} holds for $\om{P}_p^{q+1}$. Since $\Delta(\om{P}_p^q)_{q+1,i},\Delta(\om{P}_p^q)_{i,q+1}$ are independent of $j$,  \eqref{3maxdiff} gives
  \begin{multline*}
	\Delta(\om{P}_p^{q+1})_{j,i}-\Delta(\om{P}_p^{q+1})_{j-1,i}=f(\Delta(\om{P}_p^q)_{j,i},\Delta(\om{P}_p^q)_{j,q+1},\Delta(\om{P}_p^q)_{q+1,i})\\-f(\Delta(\om{P}_p^q)_{j-1,i},\Delta(\om{P}_p^q)_{j-1,q+1},\Delta(\om{P}_p^q)_{q+1,i}).
 \end{multline*}
By assumption, condition \eqref{cond1} holds for $\om{P}_p^q$ and the result follows directly from definition \eqref{2supermax}
unless $\Delta(\om{P}_p^q)_{j,q+1}\geq0$, $\Delta(\om{P}_p^q)_{j-1,q+1}<0$, and  
 \[
  	 \Delta(\om{P}_p^{q+1})_{j,i}-\Delta(\om{P}_p^{q+1})_{j-1,i}=\Delta(\om{P}_p^q)_{j,q+1}+\Delta(\om{P}_p^q)_{q+1,i}-\Delta(\om{P}_p^q)_{j-1,i}.
 \]
 The conditions $\Delta(\om{P}_p^q)_{j-1,q+1}<0$ and \eqref{cond1}, yields that $\Delta(\om{P}_p^q)_{j,q+1}<\delta_j$.
Since $j-1\geq q+1$ the non-decreasing property of $f_0$ implies that $\Delta(\om{P}_p^q)_{q+1,i}-\Delta(\om{P}_p^q)_{j-1,i}\leq0$ or $\Delta(\om{P}_p^q)_{q+1,i}<0$. In the first case we have
 \[
 	\Delta(\om{P}_p^{q+1})_{j,i}-\Delta(\om{P}_p^{q+1})_{j-1,i}\leq \Delta(\om{P}_p^q)_{j,q+1}\leq \delta_j.
\]
 In the latter case the inequality $\Delta(\om{P}_p^{q+1})_{j,i}\leq \delta_j-1$ holds.
 Hence, condition \eqref{cond1} holds for $\Delta(\om{P}_p^{q+1})_{j,i}$.

 Assume that the condition \eqref{cond2} holds for $\om{P}_p^q$. If $\Delta(\om{P}_p^q)_{j,q+1}<0$, then \eqref{cond2} holds trivially for $\Delta(\om{P}_p^{q+1})_{j,i}$.    Otherwise, it holds that
 \[
  \Delta(\om{P}_p^{q+1})_{j,i}\leq \max(\Delta(\om{P}_p^q)_{j,i},\Delta(\om{P}_p^q)_{j,q+1}+\Delta(\om{P}_p^q)_{q+1,i}).
\]
   
 Assume $i<q+1$. If $\Delta(\om{P}_p^q)_{q+1,i}\geq0$ it follows from condition \eqref{cond2} that $\Delta(\om{P}_p^q)_{q+1,i}\leq\Delta(\om{P})_{q+1,1}-(p+2)$.
  Condition \eqref{cond1} and  $\Delta(\om{P})_{q+1,i}\geq 0$ implies that $\Delta(\om{P}_p^q)_{j,q+1}\leq\Delta(\om{P})_{j,1}-\Delta(\om{P})_{q+1,1}$. Hence, $\Delta(\om{P}_p^{q+1})_{j,i}\leq  \max(\Delta(\om{P})_{j,1}-(p+2),-1)$. Otherwise, $\Delta(\om{P}_p^q)_{q+1,i}<0$, and condition \eqref{cond2} gives
   \[
   	\Delta(\om{P}_p^q)_{j,q+1}\leq \max(\Delta(\om{P})_{j,1}-(p+1),-1).
   \]
   Thus $\Delta(\om{P}_p^q)_{j,q+1}+\Delta(\om{P}_p^q)_{q+1,i}\leq \max(\Delta(\om{P})_{j,1}-(p+2),-1)$.
   
    Assume $i>q+1$. If $\Delta(\om{P}_p^q)_{q+1,i}\geq0$ it follows from condition \eqref{cond2} that $\Delta(\om{P}_p^q)_{q+1,i}\leq\Delta(\om{P})_{q+1,1}-(p+1)$.
  Condition \eqref{cond1} and  $\Delta(\om{P})_{q+1,i}\geq 0$ implies $\Delta(\om{P}_p^q)_{j,q+1}\leq\Delta(\om{P})_{j,1}-\Delta(\om{P})_{q+1,1}$. Hence, $\Delta(\om{P}_p^{q+1})_{j,i}\leq  \max(\Delta(\om{P})_{j,1}-(p+1),-1)$. Otherwise, $\Delta(\om{P}_p^q)_{q+1,i}<0$, and condition \eqref{cond2} gives
   \[
   	\Delta(\om{P}_p^q)_{j,q+1}\leq \max(\Delta(\om{P})_{j,1}-(p+1),-1).
   \]
   Thus $\Delta(\om{P}_p^q)_{j,q+1}+\Delta(\om{P}_p^q)_{q+1,i}\leq \max(\Delta(\om{P})_{j,1}-(p+1),-1)$. Hence condition \eqref{cond2} is satisfied.
  
 Assume $q=k-1$. Then we show the conditions \eqref{cond1}, \eqref{cond2} for $\om{P}^{1}_{p+1}:=\mathcal{K}_{2:k,1}(\om{P}^{k+1}_{p})\om{P}^{k+1}_{p}$. This is done similarly as for $q<k-1$ with the exception that $i>1$, which implies that only one case has to be considered in \eqref{cond2}. 

In conclusion, $\Delta(P_{d-2}^{k-1})_{j,i}\leq0$ holds for $k\geq j>i$ due to condition \eqref{cond2} and for $j<i\leq k$ the inequality holds since $f_0$ is non-decreasing in the first two arguments.
By definition we have $\widehat{\om{P}}=\mathcal{K}_{1,k,k-1}\circ\hdots\circ\mathcal{K}_{1:k,1}(P_{d-2}^{k-1})P_{d-2}^{k-1}$, which satisfies the conditions in the theorem. 
\end{proof}

The following propositions present two algorithms that for given operator matrix polynomial $\om{P}$ generates an equivalent operator matrix polynomial $\widehat{\om{P}}$, where the highest degrees are in the diagonal. The algorithm in Proposition \ref{polyrref} usually preserves a greater number of the original operator polynomial entries and exploits the structure of $\om{P}$. However, it is only applicable when $\Hs_i\simeq\Hs_j$ for $i,j\in\{1,\dots,n\}$. In the algorithms presented in Propositions \ref{polyrref}  and \ref{polyrref2},  $J_{i,j}$  denote the operator matrix permuting the rows of entries $i$ and $j$.

\begin{prop}\label{polyrref}
Let  $\om{P}$ be defined as \eqref{2genP2} and assume that $\Hs_i=\Hs_j$ for $i,j\in\{1,\dots,n\}$. Define the algorithm:

	\begin{enumerate}
	\item Set $\om{P}_1:=\om{P}$, $E_1:=I$, and $k:=1$. 
	\item If $k=n$, set $\om{P}'_k:=\om{P}_k$ and  $E'_k:=E_k$. Else, let $i\geq k$ be the least index such that $\Delta(\om{P}_k)_{i,k}\geq\Delta(\om{P}_k)_{l,k}$ for all $l\geq k$. 
	Set $\om{P}'_k:=\mathcal{K}_{k+1:n,k}(J_{k,i}\om{P}_k)J_{k,i}\om{P}_k$ and  $E'_k:=\mathcal{K}_{k+1:n,k}(J_{k,i}\om{P}_k)J_{k,i}E_k$.
	\item Set $\widecheck{\om{P}}_k:=J_{1,k}\om{P}'_kJ_{1,k}$ and  $\widecheck{E}_k:=J_{1,k}E_k'$.
	\item Let $J$ be the operator matrix that permutes the $2,\hdots,k$ diagonal operators in $\widecheck{\om{P}}_k$ to obtain $\widetilde{\om{P}}_k:=J\widecheck{\om{P}}
	_kJ^{-1}$, which satisfies  $\Delta(\widetilde{\om{P}}_k)_{i,1}\leq\Delta(\widetilde{\om{P}}_k)_{j,1}$ for all $j>i>1$ and define  $\widetilde{E}_k:=J\widecheck{E}_k$.
	\item  Obtain $\widehat{E}$ and $\widehat{\om{P}}_{k}$ by applying Lemma \ref{2diaglem} on $\widetilde{\om{P}}_{k}$ and set  $\widehat{E}_k:=\widehat{E}		\widetilde{E}_k$.
	\item Set $\om{P}_{k+1}:=J_{1,k}J^{-1}\widehat{\om{P}}_kJJ_{1,k}$ and $E_{k+1}=J_{1,k}J^{-1} \widehat{E}_k$.
	\item If $k=n$ set $\widehat{\om{P}}:=\om{P}_{k+1}$, $E:=E_{k+1}$ and terminate. Else set $k:=k+1$ and return to $(2)$.
	\end{enumerate}

By applying the algorithm to $\om{P}$,  we obtain operator matrix functions $\widehat{\om{P}}:\C\rightarrow\Ls(\Hs_1^n)$ and an invertible  $E:\C\rightarrow\Bs(\Hs_1^n)$ such that
\[
E(\lambda)\om{P}(\lambda)=\widehat{\om{P}}(\lambda)=\begin{bmatrix}
\widehat{P}_{1,1}(\lambda) &\hdots&\widehat{P}_{1,n}(\lambda) \\
\vdots & \ddots &\vdots\\
\widehat{P}_{n,1}(\lambda) & \hdots &\widehat{P}_{n,n}(\lambda)
\end{bmatrix}, \quad \lambda\in\C,
\]
where $\deg \widehat{P}_{i,i}>\deg \widehat{P}_{j,i}$ for $i\neq j$. \\
\end{prop}
\begin{proof}
The result holds trivially for $k=1$ and the proof for $k>1$ is by induction. In the inductive step we show that $\om{P}_k=E_k\om{P}$ and $\Delta(\om{P}_k)_{j,i}<\Delta(\om{P}_k)_{i,i}$ for all $j\in\{1,\dots,n\}$, $i\in\{1,\dots,k-1\}$, and $j\neq i$.
 
Assume that induction hypothesis holds for $k\geq1$.
By applying step 2 it follows that $\om{P}'_k=E'_k\om{P}$. Further since $\Delta(J_{k,i}\om{P}_k)_{k,k}\geq \Delta(J_{k,i}\om{P}_k)_{l,k}$, the condition $\Delta(J_{k,i}\om{P}_k)_{j,i}<0$ for $j>k$ and $i\leq k$ implies 
the condition $\Delta(\om{P}_k)'_{j,i}<0$ for $j>k$ and $i\leq k$. 
After step 3 we have $\widecheck{\om{P}}_k=\widecheck{E}_k\om{P}J_{1,k}$ and the inequality  $\Delta(\widecheck{\om{P}}_k)_{j,i}<\Delta(\widecheck{\om{P}}_k)_{i,i}$  holds for all $j\in\{1,\dots,n\}$ and $i\in\{2,\dots,k\}$, since the $k$-th column is swapped with column one.\\
The existence of $J$ in step 4 is obvious and from the definitions $\widetilde{\om{P}}_k=\widetilde{E}_k\om{P}J_{1,k}J^{-1}$ and $\Delta(\widetilde{\om{P}}_k)_{j,i}<\Delta(\widetilde{\om{P}}_k)_{i,i}$ for all $j\in\{1,\dots,n\}$ and $i\in\{2,\dots,k\}$.\\
By construction  $\widetilde{\om{P}}_k$ satisfies the assumptions of Lemma \ref{2diaglem}. This lemma then implies that $\widehat{\om{P}}_k=\widehat{E}_k\om{P}J_{1,k}J^{-1}$ and $\Delta(\widetilde{\om{P}}_k)_{j,i}<\Delta(\widetilde{\om{P}}_k)_{i,i}$ for all $j\in\{1,\dots,n\}$ and $i\in\{1,\dots,k\}$. 

Hence, $\widehat{\om{P}}_k$ satisfies the desired condition for $\om{P}_{k+1}$, but the equivalence is $\widehat{\om{P}}_k=\widehat{E}_k\om{P}J_{1,k}J^{-1}$. Step $6$ finds an equivalence of the desired type, $\om{P}_{k+1}=E_{k+1}\om{P}$ and since $J_{1,k}J^{-1}$ is a permutation operator matrix of first $k$ rows the condition $\Delta(\widetilde{\om{P}}_k)_{j,i}<\Delta(\widetilde{\om{P}}_k)_{i,i}$ for all $j\in\{1,\dots,n\}$, $i\in\{1,\dots,k\}$ and $i\neq j$ implies the same conditions for $\om{P}_{k+1}$. Hence, the result follows by induction. 
\end{proof}

\begin{prop}\label{polyrref2}
Let  $\om{P}$ be defined as \eqref{2genP2} and define the algorithm: 

	\begin{enumerate}
	\item Set $\om{P}_2:=\om{P}$, $E_2:=I$, and $k:=2$. 
	\item Obtain $E$ and $\om{P}_{k}'$ by applying Lemma \ref{2rowlem} on $\om{P}_{k}$ and set  $E_k':=E{E}_k$.
	\item Set $\widecheck{\om{P}}_k:=J_{1,k}\om{P}'_kJ_{1,k}$ and  $\widecheck{E}_k:=J_{1,k}E_k'$.
	\item Let $J$ be the operator matrix that permutes the $2,\hdots,k$ diagonal operators in $\widecheck{\om{P}}_k$ to obtain $\widetilde{\om{P}}_k:=J\widecheck{\om{P}}
	_kJ^{-1}$, which satisfies  $\Delta(\widetilde{\om{P}}_k)_{i,1}\leq\Delta(\widetilde{\om{P}}_k)_{j,1}$ for all $j>i>1$ and define  $\widetilde{E}_k:=J\widecheck{E}_k$.
	\item Obtain $\widehat{E}$ and $\widehat{\om{P}}_{k}$ by applying Lemma \ref{2diaglem} on $\widetilde{\om{P}}_{k}$ and set  $\widehat{E}_k:=\widehat{E}		\widetilde{E}_k$.
	\item Set $\om{P}_{k+1}:=J_{1,k}J^{-1}\widehat{\om{P}}_kJJ_{1,k}$ and $E_{k+1}=J_{1,k}J^{-1} \widehat{E}_k$.
	\item If $k=n$ set $\widehat{\om{P}}:=\om{P}_{k+1}$, $E:=E_{k+1}$ and terminate. Else set $k:=k+1$ and return to $(2)$.
	\end{enumerate}

By applying the algorithm to $\om{P}$,  we obtain operator matrix functions $\widehat{\om{P}}:\C\rightarrow\Ls(\Hs_1\oplus\hdots \oplus\Hs_n)$ and an invertible $E:\C\rightarrow\Bs(\Hs_1\oplus\hdots \oplus\Hs_n)$ such that
\[
E(\lambda)\om{P}(\lambda)=\widehat{\om{P}}(\lambda)=\begin{bmatrix}
\widehat{P}_{1,1}(\lambda) &\hdots&\widehat{P}_{1,n}(\lambda) \\
\vdots & \ddots &\vdots\\
\widehat{P}_{n,1}(\lambda) & \hdots &\widehat{P}_{n,n}(\lambda)
\end{bmatrix}, \quad \lambda\in\C,
\]
where $\deg \widehat{P}_{i,i}>\deg \widehat{P}_{j,i}$ for $i\neq j$. \\
\end{prop}
\begin{proof}
The proof is by induction, where we show that $\om{P}_k=E_k\om{P}$ and $\Delta(\om{P}_k)_{j,i}<\Delta(\om{P}_k)_{i,i}$ for all $j\in\{1,\dots,k-1\}$ and $i\in\{1,\dots,k-1\}$ such that $i\neq j$. The basis $\om{P}_2$ follows from definition and the proof of the induction step is very similar to the induction in Proposition \ref{polyrref}. The only difference is in step $2$, where Lemma \ref{2rowlem} is used. 
\end{proof}

\begin{rem}
Despite Proposition \ref{polyrref2} it is important to realize that when $\Hs_i\neq\Hs_j$ for some $i,j$, additional problems might occur. 
For example, consider the operator matrix polynomial $\om{P}:\C\rightarrow\Ls(\Hs\oplus\HS)$, defined as
\[
	\om{P}(\lambda)=\begin{bmatrix}
	A-\lambda &B\lambda \\
	C\lambda ^2 & D-\lambda
	\end{bmatrix},\quad \lambda\in\C.
\]
Define $\widehat{\om{P}}:\C\rightarrow\Ls(\Hs\oplus\HS)$ as
\[
	\widehat{\om{P}}(\lambda):=\mathcal{K}_{2,1}(\om{P})\om{P}(\lambda):=\begin{bmatrix}
	A-\lambda &B\lambda \\
	 CA^2&D+(CAB-I_{\HS})\lambda+CB\lambda^2
	\end{bmatrix}.
\]
$\widehat{\om{P}}(\lambda)$ has the form assumed in Theorem \ref{polygenlin}, but the highest order in the $(2,2)$-th entry, $CB$, might be degenerate for all operators $C$ and $B$ regardless if $D$ is invertible or not.
\end{rem}

By combining the results in Theorem \ref{2schurrat0}, Theorem \ref{3jordigblo}, Theorem \ref{polygenlin}, and Proposition \ref{polyrref2} (or Proposition \ref{polyrref}) we obtain a method of linearizing a class of operator matrix functions. This class consists of operator matrices where, each entry is a product and/or Schur complement  of polynomials and the method extends the applicability of linearization to a larger class compared with a method based on the results in Section $3$ alone. An illustrative example is presented in the following subsection.  
 
\subsection{Example of linearization of an operator matrix function}\label{sec:ex}
Let $M,N_i\in\Bs(\Hs)$ for $i=0,1,2,3$, $A\in\Bs(\Hs,\HS)$, $C_i\in\Ls(\Hs,\HS)$ for $i=0,1,2$, $D_0\in\Ls(\HS)$, $B,D_1,D_2,Q\in\Bs(\HS)$, and $P_0,P_1\in \Ls(\HS,\Hs)$. Further assume that there is a $j$ and an $l$ such that $C_i\in\Bs(\Hs,\HS)$ for $i\neq j$ and $P_i\in\Bs(\HS,\Hs)$ for $i\neq l$. Let $D:\C\rightarrow\Ls(\HS)$ be defined  as $D(\lambda)=D_2\lambda^2+D_1\lambda+D_0$, $\lambda\in\C$. If $j=l=0$ let $\Cdom:=\rho(D)$ else  $\Cdom:=\rho(D)\setminus \{0\}$. Finally assume that $D^{-1}(\lambda)C_j$ for $\lambda\in\Cdom$ is bounded on $\dom (C_j)$, which is dense in $\Hs$ and $N_3$, and $D_2Q$ are invertible operators.

In each step the operator matrix function is defined on its natural domain.
Consider the operator matrix function $\om{S}:\Cdom\rightarrow\Ls(\Hs\oplus\HS)$, 
\[\om{S}(\lambda)=
	\begin{bmatrix}
	 (M-\lambda)(N_3\lambda^3+N_2\lambda^2+N_1\lambda+N_0)&P_1\lambda+P_0&\\
	 A\lambda-(B-\lambda)D^{-1}(\lambda)(C_2\lambda^2+C_1\lambda+C_0)&Q\lambda
	\end{bmatrix}.
\]
This function can be linearized by the following steps:

Theorem \ref{3jordigblo} states that after $I_{\Hs}$-extension $\om{S}$ is equivalent to $\widehat{\om{S}}:\Cdom\rightarrow\Ls(\Hs^2\oplus\HS)$, 
\[\widehat{\om{S}}(\lambda):=
	\begin{bmatrix}
	 M-\lambda&&P_1\lambda+P_0\\
	 -I &N_3\lambda^3+N_2\lambda^2+N_1\lambda+N_0\\
	 &A\lambda-(B-\lambda)D^{-1}(\lambda)(C_2\lambda^2+C_1\lambda+C_0)&Q\lambda
	\end{bmatrix}.
\]
Theorem \ref{2schurrat0} states that  $\widehat{\om{S}}$ is after $D$-extension equivalent to $\om{P}:\Cdom\rightarrow\Ls(\Hs^2\oplus\HS^2)$, 
\[\om{P}(\lambda):=
	\begin{bmatrix}
	 M-\lambda&&&P_1\lambda+P_0\\
	 -I_\Hs &N_3\lambda^3+N_2\lambda^2+N_1\lambda+N_0\\
	 &A\lambda&B-\lambda&Q\lambda\\
	 & C_2\lambda^2+C_1\lambda+C_0&D(\lambda)
	\end{bmatrix}.
\]

$\om{P}$ is an operator matrix polynomial, but in the last two columns the highest degree is not strictly in the diagonal. 
Hence, an equivalent problem has to be found. Apply the algorithm given in Proposition \ref{polyrref2} to $\om{P}$. This results in the equivalent operator function $\widehat{\om{P}}:=\mathcal{K}_{4,3}(\om{P})\om{P}$, 
\[\widehat{\om{P}}(\lambda)=
	\begin{bmatrix}
	 M-\lambda&&&P_1\lambda+P_0\\
	 -I_\Hs &N_3\lambda^3+N_2\lambda^2+N_1\lambda+N_0\\
	 &A\lambda&B-\lambda&Q\lambda\\
	 & G\lambda^2+(C_1+KA)\lambda+C_0&D_B& D_2Q\lambda^2+KQ\lambda
	\end{bmatrix},
\]
where $G=C_2+D_2A$, $\dom(G)=\dom(C_2)$, $D_B:=D_2B^2+D_1B+D_0$, $\dom(D_B)=\dom(D_0)$, and $K:=D_1+D_2B$. In $\widehat{\om{P}}$
the highest degrees are in the diagonal and at most one coefficient in $G\lambda^2+(C_1+KA)\lambda+C_0$ and $P_1\lambda+P_0$ are unbounded. Hence, Theorem \ref{polygenlin} can be applied. Define $\mon{G}:=(D_2Q)^{-1}G$, $\mon{K}:=(D_2Q)^{-1}K$, $\mon{C}_i:=(D_2Q)^{-1}C_i,$ and $\mon{D}_B:=(D_2Q)^{-1}D_B$. Let  $\om{W}$ denote the function defined in Theorem \ref{polygenlin}. Then is $\widehat{\om{P}}(\lambda)$ after $\om{W}(\lambda)$-extension equivalent to $\om{T}-\lambda$ on $\Cdom$, where the operator matrix $\om{T}\in\Ls(\Hs^4\oplus\HS^3)$ is defined as
\[
\om{T}:=
	\begin{bmatrix}
	 M&&&& &P_1&P_0\\
	 N_3^{-1} &-N_3^{-1}N_2 &-N_3^{-1}N_1&-N_3^{-1}N_0\\
	& I_{\Hs}& \\
	 && I_{\Hs} &\\
	 & & A & & B&Q\\
	 & -\mon{G}&-\mon{C}_1-\mon{K}A&-\mon{C}_0&-\mon{D}_B& -\mon{K}Q\\
	 & & &&  &   I_{\HS}&
	 	\end{bmatrix}.
\]
In conclusion, $\om{S}(\lambda)$ is after $I_{\Hs}\oplus D(\lambda)\oplus\om{W}(\lambda)$-extension equivalent to $\om{T}-\lambda$ for all $\lambda\in\Cdom$. Hence, Proposition \ref{2defcor} yields that the spectral properties of $\om{T}$ and of $\om{S}$ coincides.

\vspace{3.5mm}

{\small
{\bf Acknowledgements.} \ 
The authors gratefully acknowledge the support of the Swedish Research Council under Grant No.\ $621$-$2012$-$3863$. We sincerely thank the reviewer for 
the insightful comments, which were invaluable when revising the manuscript.
}

\bibliographystyle{alpha}
\bibliography{BIB}

\end{document}